\documentclass[a4paper,11pt]{amsart}
\usepackage[english]{babel}
\usepackage{amsmath,amsfonts,amssymb,amsthm,relsize,setspace,nicefrac, yhmath, amscd,eucal}
\usepackage[left=1.5cm,right=1.5cm,top=2.0cm,bottom=2.5cm]{geometry}
\usepackage{enumitem}
\usepackage{dsfont}
\usepackage{bm}
\usepackage{amsmath}
\usepackage{amssymb}                 
\usepackage{xcolor}
\usepackage{hyperref} 

\newtheorem{theorem}{Theorem}[section]
\newtheorem{corollary}[theorem]{Corollary}
\newtheorem{lemma}[theorem]{Lemma}
\newtheorem{proposition}[theorem]{Proposition}

\theoremstyle{definition}

\newtheorem{remark}{Remark}

\newcommand{\be}{\begin{equation}}
	\newcommand{\bel}[1]{\begin{equation}\label{#1}}
		\newcommand{\ee}{\end{equation}}
	
	\newcommand{\barr}{\begin{eqnarray}}
		\newcommand{\earr}{\end{eqnarray}}
	\newcommand{\bars}{\begin{eqnarray*}}
		\newcommand{\ears}{\end{eqnarray*}}
	
	\newtheorem{subn}{\name}
	
	\newcommand{\bsn}[1]{\def\name{#1}\begin{subn}}
		\newcommand{\esn}{\end{subn}}

	\newtheorem{sub}{\name}[section]

	\newcommand{\bs}{\begin{sub}}
		\newcommand{\es}{\end{sub}}
	
	\newcommand{\bth}[1]{\def\name{Theorem}
		\begin{sub}\label{t:#1}}
		\newcommand{\blemma}[1]{\def\name{Lemma}
			\begin{sub}\label{l:#1}}
			\newcommand{\bcor}[1]{\def\name{Corollary}
				\begin{sub}\label{c:#1}}
				\newcommand{\bdef}[1]{\def\name{Definition}
					\begin{sub}\label{d:#1}}
					\newcommand{\bprop}[1]{\def\name{Proposition}
						\begin{sub}\label{p:#1}}
						
						\newcommand{\BA}{\begin{array}}
							\newcommand{\EA}{\end{array}}
						\newcommand{\BAN}{\renewcommand{\arraystretch}{1.2}
							\setlength{\arraycolsep}{2pt}\begin{array}}
							\newcommand{\BAV}[2]{\renewcommand{\arraystretch}{#1}
								\setlength{\arraycolsep}{#2}\begin{array}}
								\newcommand{\BSA}{\begin{subarray}}
									\newcommand{\ESA}{\end{subarray}}
								
								\newcommand{\BAL}{\begin{aligned}}
									\newcommand{\EAL}{\end{aligned}}
								\newcommand{\BALG}{\begin{alignat}}
									\newcommand{\EALG}{\end{alignat}}
								\newcommand{\BALGN}{\begin{alignat*}}
									\newcommand{\EALGN}{\end{alignat*}}
								\def\angb<#1>{\langle #1 \rangle}

								\newcommand{\supp}{\opname{supp}}

								\def\N{\mathbb{N}}

								\def\R{\mathbb{R}}

								\let\.=\cdot
								\let\0=\emptyset

								\def\O{\Omega}

								\def\supp{\text{\rm supp}}

								\numberwithin{equation}{section}

								\theoremstyle{definition}

								\let\.=\cdot
								\let\0=\emptyset

								\def\O{\Omega}

								\def\supp{\text{\rm supp}}

								\newenvironment{formula}[1]{\begin{equation}\label{eq:#1}}
									{\end{equation}\noindent}
								
								\def\Fi#1{\begin{formula}{#1}}
									\def\Ff{\end{formula}\noindent}
								
								\usepackage{wrapfig}

\newcommand{\nlp}[3]{\|{#1}\|_{L^{#2}{(#3)}}}

\def\r{{\mathcal{R}}}

\begin{document}
\fontsize{11pt}{5pt}\selectfont


  \title[Stability  and limiting properties]{Stability  and limiting properties of  generalized principal eigenvalue for inhomogeneous nonlocal cooperative system}

\author{Ninh Van Thu}

\address{School of Applied Mathematics and Informatics,
	Hanoi University of Science and Technology, no. 1 Dai Co Viet, Hanoi, Vietnam}
\email{thu.ninhvan@hust.edu.vn}
\address{Faculty of Mathematics and Applications, Saigon University, 273 An Duong Vuong st., Ward 3, Dist.5, Ho Chi Minh City, Viet Nam}
\author{Hoang -Hung Vo$^{*}$}
\email{vhhung@sgu.edu.vn}
\thanks{$^*$ Corresponding author}


\begin{abstract} The principal eigenvalue for  linear elliptic operator has been known to be one of  very useful tools to investigate many important partial differential equations. Due to the pioneering works of Berestycki et al. \cite{BCV1,BCV2}, the study of qualitative properties for the principal eigenvalue of nonlocal operators has attracted a lot of attention of the community from theory to application (For examples \cite{LL22-1,LLS22,LZ1,LZ2,SLLY,DDF,XLR}). In this paper, motivated from the study of mathematical  modeling the dynamics of infectious diseases  in \cite{NV1, ZZLD, WD}, we analyze the asymptotic properties of the principal eigenvalue of nonlocal inhomogeneous  cooperative  system  with respect to the dispersal rate and dispersal range. This can be done thanks to the deep results of Rainer \cite{Ra13}, Kriegl and  Michor \cite{KM03} on the stability of  eigenvalue of the variable matrices of zero-order coefficients and extends many results from \cite{BCV1,LL,NV1}. Our work provides a fundamental step to investigate the nonlinear system modeling the spreading of the transmitted diseases as mentioned. 
\end{abstract}


\date{\today.}



\maketitle
\tableofcontents
\section{ Introduction and statement of the results}

Over the past decades, the study of  the equation or system with nonlocal dispersals to model natural phenomena in biology, physics and related applied science has become a contemporary trend in mathematics. One of the important topics on this trend. In the pioneering work, Berestycki et al. have proposed the following equation
\begin{equation}
u_t=J\star u-u+f(x,u)=0\quad\quad t>0, x\in\R^N
\end{equation}
where $f$ is  of Fisher-KPP type nonlinearity, to study the evolution of species that has a long range dispersal strategy. This equation been employed in several contexts, ranging from population dynamics, epidemiology combustion theory to social sciences by the observation that the intrinsic variability in the capacity of the individuals to disperse generates, at the scale of a population, a long range dispersal of the population. In this research, the eigentheory for the generalized principal eigenvalue:
\begin{equation}\label{lambdap}
J\star \varphi-\varphi+\partial_s f(x,0)\varphi+\lambda_p\varphi=0\quad\quad\textrm{in $\R^N$},
\end{equation}
plays the central role. Many qualitative properties of $\lambda_p$ in (\ref{lambdap}) such as approximation to the principal eigenvalue of second order elliptic operator, scaling limits, and asymptotic behaviors with respect to the parameters  have been carefully investigated in \cite{BCV2}.   Very recently,
the qualitative properties and asymptotic behaviors of the principal eigenvalues  with respect to the parameters for time-periodic operator have become the topics of intensive research in this community,  the interested readers are refered to \cite{baihe,LL22-1,LLS22,Liu3,Liu4,SLLY,SX,ShenVo,Vo1}. There has been a huge amount of improvements   and extensions  for applications and for purely theoretical purposes  inspired from these works, for instance the long time dynamics \cite{ CDLL, DDF,XLR,WD,KR}, the wave propagation and estimate the spreading speed   \cite{BSS19, BCHV, DWL22,Liu4,LZ1,XLR}

In the seminal works Cappasso et al. \cite{Cappasso1,Cappasso2} have proposed a partially degenerate  reaction-diffusion model to study for oro-faecal transmitted diseases in the European Mediterranean regions. This model has been strongly interested and developed in the community in both pure and applied aspects. In particular, in \cite{HY13}, Hsu and Yang studied the front propagation of the  general epidemic model
\begin{equation}\label{wave}
\left\lbrace\begin{array}{ll}
u_t=d_1u_{xx}-\alpha_1 u+h(v)\\
v_t=d_2v_{xx}-\alpha_2v+g(u)
\end{array}\right.\quad\quad t>0,x\in\R
\end{equation}
while Wu and Hsu considered this model with time delay in the deep work \cite{WuHsu}. The global dynamics for  free boundary model of (\ref{wave})  has been recently studied by many authors \cite{NV1, ZZLD, WD}, especially Nguyen and Vo \cite{NV1} investigated the long time dynamics for  
  cooperative system with free boundaries and couple nonlocal dispersals read by
 \begin{align}\label{nonlocalfree} \left\{\begin{array}{lll} u_t =  d_1\left[\displaystyle\int\limits_{g(t)}^{h(t)}J_1(x-y)u(t,y)dy - u(t,x)\right] - au(t, x) + H\left(v(t, x)\right), & t>0,\,\,\,\, \, g(t)<x<h\left(t\right), \\ v_t = d_2\left[\displaystyle\int\limits_{g(t)}^{h(t)}J_2(x-y)v(t,y)dy - v(t,x)\right] -bv(t,x)+ G\left(u(t,x)\right), & t>0,\,\,\,\, \, g(t)<x<h\left(t\right), \\ u\left(t, x\right) = v\left(t, x\right)=0, & t>0,\,\, x = g(t) \,\,\,\text{or}\,\,\, x = h(t),   \\ h^{\prime}(t) = \mu\left( \displaystyle\int\limits_{g(t)}^{h(t)}\displaystyle\int\limits_{h(t)}^{\infty}J_1(x-y)u(t, x)dydx  +\rho\displaystyle\int\limits_{g(t)}^{h(t)}\displaystyle\int\limits_{h(t)}^{\infty}J_2(x-y)v(t, x)dydx\right),& t>0,\\ g^{\prime}(t) = -\mu\left( \displaystyle\int\limits_{g(t)}^{h(t)}\displaystyle\int\limits_{-\infty}^{g(t)}J_1(x-y)u(t, x)dydx+\rho\displaystyle\int\limits_{g(t)}^{h(t)}\displaystyle\int\limits_{-\infty}^{g(t)}J_2(x-y)v(t,x)dydx\right),& t>0,\\ -g(0) = h(0) = h_0,\,\, u(0, x) = u_0(x),\,\,v(0, x) = v_0(x),& x\in \left[-h_0, h_0\right], \end{array}\right. 
\end{align}
where $x = g(t)$ and $x = h(t)$ are the moving boundaries to be determined together with $u(t, x)$ and $v(t, x)$, which are always assumed to be {{identical}} $0$ for $x \in \R\setminus\left[g(t), h(t)\right]$; $a, b, d_1, d_2, \mu, \rho$ given real positive constants. The main tool in the work \cite{NV1} is the eigentheory for linear cooperative system associated with homogeneous matrix of zero order coefficient. However, in reality, the zero order coefficients $a, b$ and the functions $H,G$ should depend on  spatial variables. Thus, it is much more resonable to consider the system (\ref{nonlocalfree}) when zero order coefficients depend also on variable $x$. 

The existence, simplicity and qualitative properties of the principal eigenvalue for the cooperative system has been  the subject of intensive interest since a long time ago, the earliest work on this topic may be due to  Hess \cite{Hess}. Then, Cantrell et al. \cite{C1, CS, CC} used  several parameter bifurcation theory to investigate several deep qualitative and quantitative information about the generalised spectrum for second-order elliptic systems. Recently, in the interesting works \cite{CLG, LL, baihe}, Caudevilla and L\'opez-G\'omez, Lam and Lou, also Bai and He  obtained the asymptotic behaviors of the principal eigenvalue with respect to the parameters associated with the diffusion or zero order coefficients, which imply many plentiful phenomena in the population dynamics. In this paper, motivated from the mentioned works,  our main goal is to investigate the asymptotic properties  of the principal eigenvalue $\lambda_p$ with respect to the dispersal rate $d_i$, $i\in\overline{1,N}$ and dispersal range $\sigma$ of the nonlocal Dirichlet-type system:
\begin{align}\label{eigenproblem1}
	\left\{\begin{array}{ll}
	\dfrac{d_1}{\sigma^m}\left[\displaystyle\int_\Omega J_{1,\sigma}(x-y) \varphi_1(y)dy-\varphi_1\right]+ a_{11}(x)\varphi_1(x)+\cdots + a_{1n}(x)\varphi_N(x)+\lambda_p\varphi_1(x)&=0 \\
	\vdots \hskip 3cm\vdots \hskip 3cm\vdots &\vdots  \,\\
	 \dfrac{d_N}{\sigma^m}\left[\displaystyle\int_\Omega J_{N,\sigma}(x-y) \varphi_N(y)dy-\varphi_N\right]+ a_{N1}(x)\varphi_1(x)+\cdots+a_{NN}(x)\varphi_N(x)+\lambda_p\varphi_N(x)&=0
	\end{array}x \in \Omega,\right.
	\end{align}
where $\Omega\subset\R^n$ is a bounded domain, $J_{\sigma, i}(x):=\frac{1}{\sigma^n} J_i(\frac{x}{\sigma})$, $1\leq i\leq N$ with $N$ is a number of equations. 

For the sake of presentation, we first define the function spaces that are used throughout the paper:

\begin{eqnarray}
\mathbf E=\mathbf{E}(\Omega)& =& L^2\left(\Omega\right)\times L^2\left(\Omega\right)\times \cdots\times L^2\left(\Omega\right)\nonumber ;\\
\mathbf{E}^\infty(\Omega) &=& L^\infty\left(\Omega\right)\times L^\infty\left(\Omega\right)\times \cdots\times L^\infty\left(\Omega\right)\nonumber ;\\
\mathbf C:=\mathbf{C}(\Omega)&=&\mathcal{C}\left(\Omega\right)\times \mathcal{C}\left(\Omega\right)\times\cdots\times \mathcal{C}\left(\Omega\right);\nonumber\\
\mathbf{C}_c(\Omega)&=&\mathcal{C}_c\left(\Omega\right)\times \mathcal{C}_c\left(\Omega\right)\times\cdots\times \mathcal{C}_c\left(\Omega\right)\nonumber ;\\
\mathbf{E}^+&=&\left\{\pmb{\varphi}=(\varphi_1,\varphi_2,\ldots, \varphi_N)^T\in \mathbf{E}\; \text{ such that } \varphi_j\geq 0, 1\leq j\leq N\right\}\nonumber ;\\
\mathbf{E}^{++}&=&\left\{\pmb{\varphi}=(\varphi_1,\varphi_2,\ldots, \varphi_N)^T\in \mathbf{E}\; \text{ such that } \varphi_j> 0, 1\leq j\leq N \right\},\nonumber
\end{eqnarray}
and we denote by $u{\bf \geq} v$ if $u=(u_1,u_2,...,u_N)$, $v=(v_1,v_2,...,v_N)$ and $u_j\geq v_j$ for $1\leq j\leq N$, $I$ the unit matrix, and $\left(\mathcal{C}(\Omega)\right)^{N\times N}$ the set of continuous $N\times N$ matrix-valued function on $\Omega$.  Note that $\mathbf{E}$ is Hilbert space with inner product
	\begin{align*}
	\left \langle {\varphi} ,{\phi} \right \rangle=\displaystyle\int_\Omega\varphi_1(x)\phi_1(x)dx+\cdots+\displaystyle\int_\Omega\varphi_N(x)\phi_N(x)dx,
	\end{align*}
	where ${\varphi}=\left(\varphi_1,\ldots,\varphi_N\right)^T,$  ${\phi}=\left(\phi_1,\ldots,\phi_N\right)^T\in \mathbf{E}$. We also denote by $\bm{\mathcal A}\colon \mathbf E\to \mathbf C$ given by $(\bm{\mathcal A}\varphi)(x)=A(x)\varphi(x)$ and  $\left(\bm{\mathcal I}\varphi\right)(x)=I\varphi(x)=\varphi(x)$.

 Throughout the paper, we always assume, if no other specifically mentioned, that the matrix $A(x)=(a_{ij}(x))\in (\mathcal{C}(\overline{\Omega}))^{N\times N}$ satisfies $a_{ij}(x)=a_{ji}(x)>0$ for any $1\leq i, j\leq N$ and for all $x\in \overline{\Omega}$. The dispersal kernel functions $J_1, \ldots, J_N\colon \mathbb R^n\to \mathbb R\ (N=1,2,\ldots)$  satisfy the following assumption:
\begin{align*}
	\begin{array}{lll}
		({\bf J})\; J_i\in C\left(\mathbb R^n\right)\bigcap L^{\infty}\left(\mathbb R^n\right)\; \text{is nonnegative, symmetric}\; \text{such that}\;\; J_i(0) > 0, \; \displaystyle\int_{\mathbb R^n} J_i(x)dx = 1,\; i=1,\ldots, N.
	\end{array}
\end{align*}

We define $\bm{\mathcal N}\colon \mathbf{E}\rightarrow \mathbf{C}$ given by
	\begin{align*}
	\left(\bm{\mathcal N}\varphi\right)(x)=\mathrm{diag}\left({\mathcal N}_1[\varphi_1](x),\ldots,\; {\mathcal N}_N[\varphi_N](x)\right),
	\end{align*}
	where ${\mathcal N}_i[\varphi_i](x): = 
	\displaystyle\int_\Omega J_i(x-y)\varphi_i(y)dy, \;\text{for}\; i=1,\ldots, N$. In what follows, $\bm{\mathcal N}_\delta$ denotes the $\bm{\mathcal N}$ associated to $J_{\sigma,i}$, $1\leq i\leq N$ and $\Omega_\delta$, where $J_{\sigma, i}(x-y):=\frac{1}{\sigma^N} J_i(\frac{x-y}{\sigma})$ and $\O_\sigma:=\frac{1}{\sigma}\O$. Then the eigenvalue problem we consider is the real number $\lambda$ and a positive real eigenfunction that satisfy the system: 
\begin{align}\label{PEV}
\bm{\mathcal K}{\varphi}+\lambda {\varphi}=0,
\end{align}
where $\bm{\mathcal K}:\mathbf{E}\rightarrow \mathbf{E}$ defined by $\bm{\mathcal K}=\bm{\mathcal D}\bm{\mathcal N}-\bm{\mathcal D}+\bm{\mathcal A}$. If $\bm{\mathcal D}=(d_1,d_2,\ldots, d_N)$, $d_i>0$ for $i=1...N,$ then
the eigenvalue problem can be explicitly rewritten by

\begin{align}\label{eigenproblem}
	\left\{\begin{array}{ll}
	d_1\displaystyle\int_\Omega J_1(x-y) \varphi_1(y)dy-d_1\varphi_1+ a_{11}(x)\varphi_1(x)+\cdots + a_{1n}(x)\varphi_N(x)+\lambda\varphi_1(x)&=0 \\
	\vdots \hskip 3cm\vdots \hskip 3cm\vdots &\vdots  \,\\
	 d_N\displaystyle\int_\Omega J_N(x-y) \varphi_N(y)dy-d_N\varphi_N+ a_{N1}(x)\varphi_1(x)+\cdots+a_{NN}(x)\varphi_N(x)+\lambda\varphi_N(x)&=0 
	\end{array}\right.x \in \Omega.
	\end{align}

It is noticed that, due to the symmetry of $A(x)$, all eigenvalue functions, say $\lambda_1(x), \ldots, \lambda_N(x)$, are real. Moreover, these eigenvalue functions can be chosen to depend continuously on $x\in \overline{\Omega}$ (cf. \cite[Theorem 6]{Lax07}). Therefore, the function
 $\bar\lambda(A(x)):=\max\{\lambda_1(x), \ldots, \lambda_N(x)\}$ is continuous on $\overline{\Omega}$ and thus it must attain the maximum at some point $x_0\in \overline{\Omega}$ and without loss of generality one may assume that $\bar\lambda(A(x_0))=\lambda_1(x_0)$. Moreover, { when $A=(a_{ij})\in (\mathbb R)^{N\times N}$ be a real-valued square matrix whose off-diagonal terms are non-negative, the Perron-Frobenious theorem (cf. Theorem \ref{Perron-Frobenious}) says that there exists real eigenvalue $\bar\lambda(A)$ corresponding to a non-negative eigenvector, with the greatest real part (for any eigenvalue $\lambda'\ne \bar\lambda(A)$, $\mathrm{Re}(\lambda')<\bar\lambda(A))$. In other words, in this situation $\bar\lambda(A(x))$ is a constant function. In general, for a symmetric matrix-valued function $A(x)$ defined on $\overline{\Omega}$ we need the following hypothesis.}
 
 \textbf{Hypothesis $(\mathbf P)$}: We say that a symmetric matrix-valued function $A(x)$ defined on $\overline{\Omega}$ satisfies the Hypothesis $(\mathbf P)$ if there exist a point $x_0\in \overline{\Omega}$, $\delta>0$, and a continuous eigenpair function denoted by $\left(\lambda_1(x), e(x)\right)$ defined on $$\Omega_\delta:=\{x\in \Omega\colon \|x-x_0\|<\delta\} $$ satisfying the following conditions:
 \begin{itemize}
 \item[($\mathbf P1$)]  $\displaystyle \sigma:= \bar\lambda(A(x_0))=\lambda_1(x_0)=\sup_{x\in \overline{\Omega})}\bar\lambda(A(x))$ and $\|e(x)\|=1$ for all $x\in \Omega_\delta$.
 \item[($\mathbf P2$)] $\displaystyle \int_{\Omega_\delta}\frac{dx}{\sigma-\lambda_1(x)}=+\infty$. 
 \end{itemize}
 The Hypothesis $(\mathbf P)$ holds if $A(x)$ satisfies at least one of the following settings:
	\begin{itemize}
	\item[(a)] $A(x)$ is constant. { (See Theorem \ref{Perron-Frobenious} below}.)
	\item[(b)] $N=1$ and $A(x)$ is Lipschitz continuous on $\overline{\Omega}$ (see \cite{HMMV03}) or more generally $A(x)$ is continuous $\overline{\Omega}$ and  $\displaystyle \frac{1}{\sigma-A(x)}\not \in L^1(\overline{\Omega})$ (see \cite{Co10}).
	\item[(c)] $n=1$ and $A(x)$ is differentiable on $ \overline{\Omega}$ and has only simple eigenvalue everywhere in $\overline{\Omega}$ (cf. \cite[Theorem 8]{Lax07}).
	\item[(d)] $n=1$ and $A(x)$ is real-analytic on $ \overline{\Omega}$ (see \cite[Theorem (A)]{KM03}).
	\item[(e)] $n=1$ and $A(x)$ is $\mathcal{C}^\infty$-smooth on $ \overline{\Omega}$ and no two unequal continuously parameterized eigenvalues meet of infinite order at any $x\in \overline{\Omega}$ (see \cite[Theorem (B)]{KM03}).
	\end{itemize}
	{ We note that for any non-negative constant matrix $A=A(x)$, the existence of the principal eigenvalue was guaranteed by the Perron-Frobenious theorem (cf. Theorem \ref{Perron-Frobenious}).

Let us define the principal spectrum point $\lambda_1(\bm{\mathcal K})$ by
 $$
 \lambda_1(\bm{\mathcal K}):=\sup\{\mathrm{Re}(\lambda) \colon \lambda\in \sigma(-\bm{\mathcal K})\},
 $$
 where $\sigma(-\bm{\mathcal K})$ is the spectrum of $-\bm{\mathcal K}$. If  $\lambda_1(\bm{\mathcal K})$ is an isolated eigenvalue of $-\bm{\mathcal K}$ with eigenfuction in $\mathbf E^{+}\cap \mathbf C(\Omega)$, then it is called the \emph{principal eigenvalue} of $-\bm{\mathcal K}$.
 Now, to state our main result, let us define the generalized principal eigenvalues  $\lambda_p(\bm{\mathcal K})$, $\lambda_v(\bm{\mathcal K})$ and $\lambda_p'(\bm{\mathcal K})$ as follows:
 \begin{equation} \label{eq7-8-1}
 	\begin{split}
	 \lambda_p(\bm{\mathcal K})&=\sup\{\lambda \in \mathbb R\colon \exists \varphi \in \mathbf{C}(\overline{\Omega}), \varphi>0, \bm{\mathcal K}[\varphi](x)+\lambda \varphi(x)\leq 0 \text{ in }\Omega\},\\
 			\lambda_v(\bm{\mathcal K})&=\inf_{\varphi\in \mathbf{E}(\Omega),\varphi\not \equiv 0} -\dfrac{\langle\bm{\mathcal K}[\varphi],\varphi\rangle}{\|\varphi \|^2_{\mathbf{E}(\Omega)}},\\
 			\lambda_p'(\bm{\mathcal K})&=\inf\{\lambda \in \mathbb R\colon \exists \varphi \in \mathbf{C}(\overline{\Omega})\cap \mathbf{E}^\infty(\Omega), 0\not \equiv\varphi\geq 0, \bm{\mathcal K}[\varphi](x)+\lambda \varphi\geq 0\text{ in }\Omega\}.
 	\end{split}
 \end{equation}

In this paper, our main results read as follows:
\begin{theorem}[\textbf{Criterion for existence of the eigenvalue}]\label{existence-eigen}
Assume that $A(x)=(a_{ij}(x))\in (\mathcal{C}(\overline{\Omega}))^{N\times N}$ is a symmetric continuous matrix-valued function, defined on $\overline{\Omega}$, whose off-diagonal terms of $A(x)$ are positive and the conditions $(\mathbf J)$ and $(\mathbf P)$ hold. Then $\bm{\mathcal K}:=\bm{\mathcal N} +\bm{\mathcal A}$ is self-adjiont and has a principal eigenvalue $\lambda_p(\bm{\mathcal K})$ given by
\begin{align}\label{3.02}
\lambda_1(\bm{\mathcal K})= -\sup\limits_{\left\|{\varphi}\right\|_{\mathbf{E}} =1}\left \langle \bm{\mathcal K}{\varphi}, {\varphi} \right \rangle=\lambda_v(\bm{\mathcal K}).
\end{align}
Moreover, the maximum being attained for a strictly a positive, continuous eigenfunction, say $\phi\in \mathbf C \cap \mathbf E^{++}$, that is unique and  associated to the principal eigenvalue $\lambda_1(\bm{\mathcal K})$. Also $\sigma(\bm{\mathcal K})\subset \left(-\infty, -\lambda_1(\bm{\mathcal K})\right]$.
\end{theorem}

Although, here we only provide a sufficient condition to guarantee the existence of principal eigenvalue for the nonlocal cooperative system, but in Section 2, we can show that this condition is sharp by giving a counterexample for existence of the principal eigenvalue if the Hypothesis (P) is violated.

By  Theorem \ref{existence-eigen}, we obtain the following corollary.
 \begin{corollary}\label{cor-existence-eigen}
Assume that $A(x)=(a_{ij}(x))\in (\mathcal{C}^1([a,b]))^{N\times N}$ is a symmetric differentiable matrix-valued function, defined on $[a,b]$, whose off-diagonal terms of $A(x)$ are positive and the condition $(\mathbf J)$ holds. Then $\bm{\mathcal K}:=\bm{\mathcal N} +\bm{\mathcal A}$ is self-adjiont and has a principal eigenvalue $\lambda_p(\bm{\mathcal K})$ given by
\begin{align}\label{3.02}
\lambda_1(\bm{\mathcal K})= -\sup\limits_{\left\|{\varphi}\right\|_{\mathbf{E}} =1}\left \langle \bm{\mathcal K}{\varphi}, {\varphi} \right \rangle=\lambda_v(\bm{\mathcal K}).
\end{align}
Moreover, the maximum being attained for a strictly a positive, continuous eigenfunction, say $\phi\in \mathbf C \cap\mathbf E^{++}$, that is unique and  associated to the principal eigenvalue $\lambda_1(\bm{\mathcal K})$. Also $\sigma(\bm{\mathcal K})\subset \left(-\infty, -\lambda_1(\bm{\mathcal K})\right]$.
\end{corollary}

\vskip 1cm

\begin{theorem}[\textbf{Asymptotic behaviors with respect to dispersal rate}]\label{pro.3.6}
	Assume that the dispersal kernel $J_i$, $i=1,\ldots, N$ satisfy $\left({\bf J}\right)$, $A(x) = (a_{ij}(x))\in (\mathcal{C}(\overline{\Omega}))^{N\times N}$ be a symmetric matrix-valued function on $\overline{\Omega}$ with $a_{ij}(x)=a_{ji}(x)>0$ whenever $i\ne j$. Then the statements below about $\lambda_p\left(\bm{\mathcal D}\right)$ hold.
	\begin{enumerate}
		\item $\lambda_p\left(\bm{\mathcal D}\right)$ is a monotone increasing function in $\bm{\mathcal D}$ in the sense that $\bm{\mathcal D}\geq \bm{\mathcal D}'$ if $d_i\geq d_i'$ for $1\leq i\leq N$;\label{1}	
		\item $\lim\limits_{\bm{\mathcal D}\to \infty}\lambda_p\left(\bm{\mathcal D}\right)= +\infty$, where $\bm{\mathcal D}\to \infty$ means that $d=\min\{d_1,\ldots, d_N\}\to +\infty$;\label{2}
		\item $\lim\limits_{\bm{\mathcal D}\to 0}\lambda_p\left(\bm{\mathcal D}\right) =- \sup_{x\in \Omega}\bar\lambda(A(x))$\label{3}.
	\end{enumerate} 
\end{theorem}

\begin{theorem}[\textbf{Asymptotic behaviors with respect to dispersal rage}]\label{scaling limit}  Assume that $A(x)=(a_{ij}(x))\in (\mathcal{C}(\overline{\Omega}))^{N\times N}$ is a continuous symmetric matrix-valued function, defined on $\overline{\Omega}$, whose off-diagonal terms of $A(x)$ are positive and the conditions $(\mathbf J)$ and $(\mathbf P)$ hold. Assume also that $|z|^2J_i(z)\in L^1(\mathbb R^n)\ (1\leq i\leq N)$. Then, there hold:
\begin{itemize}
\item[i)] The case $0<m<2$:
	\begin{align*}
		& \lim_{\sigma \to 0} \lambda_p(\bm{\mathcal K}_{\sigma,m,\Omega} + \bm{\mathcal A})=-\sup_{\Omega} \bar \lambda( A(x))\\
		& \lim_{\sigma \to +\infty}\lambda_p(\bm{\mathcal K}_{\sigma,m,\Omega} + \bm{\mathcal A})=-\sup_{\Omega} \bar \lambda( A(x))
	\end{align*} 
\item[ii)] The case $m=0$:
\begin{align*}
		& \lim_{\sigma \to 0} \lambda_p(\bm{\mathcal K}_{\sigma,0,\Omega} + \bm{\mathcal A})=-\sup_{\Omega}  \bar\lambda(A(x));\\
		& \lim_{\sigma \to +\infty}\lambda_p(\bm{\mathcal K}_{\sigma,0,\Omega} + \bm{\mathcal A})=1-\sup_{\Omega}  \bar\lambda(A(x)).
	\end{align*} 
\end{itemize}	
\end{theorem}

 The paper is organised as follows. In Section \ref{sec.3}, we provide a counterexample to the existence of positive principal eigenfunction and a proof of Theorem \ref{existence-eigen} is given.  Next, a proof of Theorem \ref{pro.3.6} is intrduced in Section \ref{sec.2}. Finally, in Section \ref{bcv-section-scal} we derive the asymptotic behaviour of the principal eigenvalue under scaling and introduce a proof of Theorem \ref{scaling limit}.

\section{Existence and Variational characterizations of the principal eigenvalue}\label{sec.3}

\subsection{Square matrices whose off-diagonal terms are non-negative}
In this subsection, we consider a matrix-valued function $A(x)$ depending continuously on $x\in\overline{\Omega}$. First of all, let us recall the Perron-Frobenious Theorem (see \cite{Ga59}).
\begin{theorem}[Perron-Frobenious Theorem] \label{Perron-Frobenious}Let $A=(a_{ij})\in (\mathbb R)^{N\times N}$ be a real-valued square matrix whose off-diagonal terms are non-negative, (i.e. $a_{ij}\geq 0$ if $i\ne j$), there exists real eigenvalue $\bar\lambda(A)$ corresponding to a non-negative eigenvector, with the greatest real part (for any eigenvalue $\lambda'\ne \bar\lambda(A)$, $\mathrm{Re}(\lambda')<\bar\lambda(A))$. Moreover, if $a_{ij}>0$ for any $i\ne j$, then $\bar\lambda(A)$ is simple with strictly positive eigenvector, and it can be characterized as the unique eigenvalue corresponding to a non-negative vector.
\end{theorem}

Now, we will finish this subsection with the following lemma which is useful for the proof of Theorem \ref{scaling limit} (cf. Subsubsection \ref{SS4.2.1}).
\begin{lemma} \label{keylemma-cutoff-function}
	Let $A(x)=(a_{ij}(x))\in (\mathcal{C}(\overline{\Omega}))^{N\times N}$ be a continuous symmetric matrix-valued function on $\overline{\Omega}$ satisfying the condition $(\mathbf P1)$. Let $\alpha=\sup_{x\in \overline{\Omega}}\bar \lambda (A(x))$. Then, there exist $\{x_k\}\subset \Omega$ be a sequence, a sequence $\{r_k\}\subset\mathbb R^+$ with $r_k \to 0^+$, and a sequence of $\mathcal{C}^\infty$-smooth functions $f_k: \Omega \to R^n$ with $\supp f_n\subset B(x_k,r_k)$ and $\|f_k\|_E=1$  such that 
	\begin{align*}
	\bar \lambda(A(x_k))>\alpha-\dfrac{1}{k}	\;\text{and}\;\int_\Omega	f_k(x)^T A(x) f_k(x) dx\geq \left(\alpha-\dfrac{2}{k}\right) , \; \forall n\in \mathbb{N}.
	\end{align*}	
\end{lemma}
\begin{proof}
Since $A(x)$ satisfies the condition ($\mathbf P1$), there exists $x_0\in \overline{\Omega}$ such that $\alpha=\sup_{x\in \overline{\Omega}}\bar \lambda (A(x))=\bar \lambda (A(x_0))$. Moreover, there exist $\delta>0$ and a continuous eigenpair function denoted by $\left(\lambda_1(x), v(x)\right)$ defined on $\Omega_\delta $ such that $\lambda_1(x_0)=\alpha$. Replacing $A(x)-d.I$ for $d>0$ big enough and shrinking the domain $\Omega_\delta$ if necessary, we may assume that $\alpha>0$ and $\lambda_1(x)>0$ for all $x\in \Omega_\delta$.   

Let $\{x_k\}\subset \Omega_\delta$ be a sequence such that $\bar \lambda(A(x_k))\geq \lambda_1(x_k)>\alpha-\dfrac{1}{n}>0$ for every $n\in \mathbb N^*$. In addition, we may assume that there exists $r_k>0$ with $r_k \to 0^+$ as $n\to \infty$ such that $\bar\lambda(A(x))\geq \lambda_1(x)>\alpha -	\frac{1}{k}$ for all $x\in B(x_k, r_k)$. 
	
	Denote by
	\begin{align*}
		A_k&=\int_{B(x_k,r_k)} \lambda_1(x) \|v(x)\|^2 dx>0;\\
		N_k&=\sup_{B(x_k,r_k)} \|v(x)\|^2\\
		M_k&=2N_k\sup_{B(x_k,r_k)} \|A(x)\|^2.
	\end{align*}
Choose a sequence $\{\epsilon_k\}\subset \mathbb R^+$ with $\epsilon_k\to 0^+$ such that
	$$
	(\alpha-\frac{1}{k})\dfrac{1-\epsilon_k/2}{1+\epsilon_k}>\alpha-\frac{2}{k}, \; k\in \mathbb{N}^*.
	$$
	Now let $\{r_k'\}\subset \mathbb R^+$ be a sequence such that $0<r_k'<r_k$ and 
	$$
	\mathrm{vol}(B(x_k,r_k)\setminus B(x_k,r_k'))< \min\left\{\dfrac{\epsilon_k A_k}{N_k(\alpha-1/k)}, \dfrac{\epsilon_k A_k}{2M_k}\right\}.
	$$	
Next, we define	$\tilde f_k(x):=v(x).\chi(|x-x_k|)$, where $\chi$ is the continuous function satisfying
	\[
	\chi(t)=\begin{cases}
		1\quad \text{if}\quad |t|<r_k'\\
		0\quad \text{if}\quad |t|>r_k.
	\end{cases}
	\]
	
	Since $\lambda_1(x)\leq \|A(x)\|$ and $|f_k(x)^T A(x) \tilde f_k(x)|\leq \|A(x)\|.\|\tilde f_k (x)\|^2$, a computation shows that
	\begin{align*}
		\int_\Omega	\tilde f_k(x)^T A(x) \tilde f_k(x) dx&=	\int_{B(x_k,r_k)}	\tilde f_k(x)^T A(x) \tilde f_k(x) dx\\
		&\geq \int_{B(x_k,r_k')}	v(x)^T A(x) v(x) dx- \int_{B(x_k,r_k)\setminus B(x_k,r_k')}  \|A(x)\|.\|\tilde f_k (x)\|^2 dx\\
		& \geq \int_{B(x_k,r_k')} \lambda_1(x)\|v(x)\|^2 dx-  \int_{B(x_k,r_k)\setminus B(x_k,r_k')} \|A(x)\|.\|v(x)\|^2 dx\\
		&\geq \int_{B(x_k,r_k)} \lambda_1(x)\|v(x)\|^2 dx-\int_{B(x_k,r_k)\setminus B(x_k,r_k')} \lambda_1(x)\|v(x)\|^2 dx\\
		&-\int_{B(x_k,r_k)\setminus B(x_k,r_k')}\|A(x)\|.\|v(x)\|^2 dx\\
		&\geq \int_{B(x_k,r_k)} \lambda_1(x)\|v(x)\|^2 dx- 2\int_{B(x_k,r_k)\setminus B(x_k,r_k')}\|A(x)\|.\|v(x)\|^2 dx\\
		&\geq A_k -2\sup_{B(x_k,r_k)}\|A(x)\|\int_{B(x_k,r_k)\setminus B(x_k,r_k')} \|v(x)\|^2 dx\\
		&\geq A_k -2 N_k \sup_{B(x_k,r_k)}\|A(x)\|\mathrm{vol}(B(x_k,r_k)\setminus B(x_k,r_k'))\\
		&\geq A_k -M_k\mathrm{vol}(B(x_k,r_k)\setminus B(x_k,r_k'))\\
		&\geq A_k(1-\epsilon_k/2)
	\end{align*}	
	and 
	\begin{align*}
		\|\tilde f_k\|_E^2&=\int_\Omega	\|\tilde f_k(x)\|^2dx=	\int_{B(x_k,r_k)}	\|\tilde f_k(x)\|^2 dx\\
		&\leq \int_{B(x_k,r_k')}	\|\tilde f_k(x)\|^2 dx+\int_{B(x_k,r_k)\setminus B(x_k, r_k')}	\|\tilde f_k(x)\|^2 dx\\
		&\leq \dfrac{A_k}{\alpha -1/k}+N_k\mathrm{vol}(B(x_k,r_k)\setminus B(x_k,r_k')) \\
		&\leq \dfrac{A_k}{\alpha -1/k}+\dfrac{\epsilon_k A_k}{\alpha-1/k}= \dfrac{(1+\epsilon_k) A_k}{\alpha-1/k}.
	\end{align*}
	Therefore, one obtains that
	$$
	\dfrac{\int_\Omega	\tilde f_k(x)^T A(x) \tilde f_k(x) dx}{\int_\Omega	\|\tilde f_k(x)\|^2dx}\geq \dfrac{A_k(1-\epsilon_k/2)}{\dfrac{(1+\epsilon_k) A_k}{\alpha-1/k}}>\alpha-\frac{2}{k}.
	$$	
By using the convolution of $f_k$ with a smooth function with small support for each $n$, one can assume that $f_k$ is $\mathcal{C}^\infty$-smooth with $\supp f_k\subset B(x_k,r_k)$  for all $k\in \mathbb N^*$. Therefore, 	the desired sequene $\{f_k\}$ is finally defined by $f_k:=\tilde f_k/\|\tilde f_k\|_E $.	
\end{proof}

\subsection{Variational characterization of the principal eigenvalue}
\hspace{13cm}\newline
Let us begin with the well-definedness of the principal eigenvalue.

\begin{lemma} Assume that $\left({\bf{J}}\right)$ holds and $A(x)\in (\mathcal{C}(\overline{\Omega}))^{N\times N}$ with $a_{ij}(x)=a_{ji}(x)>0$ whenever $i\neq j$. Then the  principal eigenvalue $\lambda_p(\bm{\mathcal K})$ is well-defined. 
 \end{lemma}
 \begin{proof}
 It suffices to show that $\lambda_p(\bm{\mathcal N}+\bm{\mathcal A})$ is well-defined. Let us first show that the set
 $$
 \Lambda:=\{\lambda \in \mathbb R\colon \exists \varphi \in \mathbf{C}(\overline{\Omega}), \varphi>0, \bm{\mathcal K}[\varphi](x)+\lambda \varphi(x)\leq 0 \text{ in }\Omega\}
 $$
 is non-empty. Indeed, let us denote $K(x)=\mathrm{diag}(k_1(x),\ldots, k_N(x))$ by
 $$
 k_i(x):=\int_\Omega J_i(x-y)dy, \; i=1,\ldots, N.
 $$
 In addition, we define $\mathbf K\colon \mathbf E\to \mathbf C$ given by $\mathbf K[\varphi](x)=\left(k_1(x)\varphi_1(x),\ldots, k_N(x)\varphi_N(x)\right)^T$. Then, fix a constant function $\psi=(1,\ldots,1)>0$ one always has $\bm{\mathcal N}[\psi]=\mathbf K[\psi]$. Therefore, for any $\lambda\leq -(N |\bm{\mathcal A}|_\infty+|\mathbf K|_\infty)$ we have
 $$
 \bm{\mathcal N}[\psi](x)+(\bm{\mathcal A}+\lambda\bm{\mathcal I})[\psi](x)= (\mathbf K+\bm{\mathcal A}+\lambda\bm{\mathcal I})[\psi](x)\leq (\mathbf K+\bm{\mathcal A}-N |\bm{\mathcal A}|_\infty\bm{\mathcal I}-|\mathbf K|_\infty\bm{\mathcal I})[\psi](x)\leq 0
 $$
 for all $x\in \Omega$, where $\displaystyle |\bm{\mathcal A}|_\infty:=\sup_{x\in \overline{\Omega};1\leq i,j\leq N}|a_{ij}(x)|$ and $\displaystyle |\mathbf K|_\infty=\sup_{x\in \overline{\Omega};1\leq i\leq N}|k_{i}(x)|$. This implies that $\Lambda$ is non-empty, as desired. 
 
 Next, since $J>0$ and $\|A(x)\|$ is bounded on $\overline{\Omega}$, for any continuous positive function $\phi$ it follows that
 $$
 \bm{\mathcal N}[\psi](x)+(\bm{\mathcal A}+|\bm{\mathcal A}|_\infty\bm{\mathcal I})[\phi](x)\geq 0.
 $$
 Hence, the set $\Lambda$ has upper bound and thus $\lambda_p$ is well-defined. 
 \end{proof}

 \subsection{Counterexample}
 In this subsection, we introduce several examples of nonlocal equation where no positive bounded eigenfunction exists.
 
 In \cite[Theorem $5.1$]{Co10}, the author considered the following principal eigenvalue problem:
 \begin{equation}\label{eq-11-3-1}
 \rho \int_\Omega u dx+a(x) u=\lambda u, 
 \end{equation}
 where $\rho>0$ and $a(x)\in \mathcal{C}(\overline{\Omega})$ satisfies the condition $\displaystyle \frac{1}{\sigma-a(x)}\in L^1_{loc}(\Omega)$, where $\sigma=\max_{\overline{\Omega}}a(x)$. Then he proved that if $\displaystyle\rho \int_\Omega \frac{dx}{\sigma-a(x)}<1$, then there exists no bounded continuous positive principal eigenfunction $\phi$ to the above equation (cf. \cite[Theorem $5.1$]{Co10}). However, following his proof even no positive eigenfunction $\phi\in L^1(\Omega)$ exists. 
 
 We note that in the equation (\ref{eq-11-3-1}), by replacing $a$ by $a-\sigma$ and then dividing both sides of (\ref{eq-11-3-1}) by $\rho$, it suffices to consider the following equation 
  \begin{equation}\label{eq-11-3-2}
  \int_\Omega u dx+a(x) u=\lambda u, 
 \end{equation}
 where $a(x)\leq 0$ and $\sigma=\max_{\overline{\Omega}}a(x)=0$. In addition, we have the following theorem.
 \begin{proposition} 
 \begin{itemize}
 \item[(a)] If $\displaystyle \int_\Omega \frac{dx}{-a(x)}<1$, then there exists no bounded positive principal eigenfunction $\phi\in L^1(\Omega)$ to (\ref{eq-11-3-2}).
 \item[(b)] If $\displaystyle \int_\Omega \frac{dx}{-a(x)}\geq 1$, then for each positive function $\phi\in L^1(\Omega)$ there exists $\lambda_\phi$ such that $(\lambda_\phi, \phi)$ is an eigenpair of (\ref{eq-11-3-2}).
 \end{itemize}
 \end{proposition}
 \begin{proof}
 Let $\phi$ be a positive function $\phi\in L^1(\Omega)$ that we normalize by $\int_\Omega \phi(x) dx=1$. Then, one considers the following equation 
 $$
   \int_\Omega \phi(x) dx+a(x) \phi(x)=\lambda \phi(x), \; x\in \Omega.
 $$
 This is equivalent to 
 $$
 \phi(x)=\dfrac{1}{\lambda -a(x)}, \; x\in \Omega, 
 $$
 and hence, $\lambda >0$ and
\begin{equation}\label{eq-11-3-3}
   F(\lambda):=\int_\Omega \dfrac{dx}{\lambda -a(x)}=1.
 \end{equation}

Note that $\displaystyle F(\lambda)\leq \int_\Omega \dfrac{dx}{-a(x)}$. Therefore, if $\displaystyle \int_\Omega \dfrac{dx}{-a(x)}<1$, then there is no $\lambda \geq 0$ satisfying the equation (\ref{eq-11-3-2}), which proves $(a)$.  Thus, we consider the case that $\displaystyle \int_\Omega \dfrac{dx}{-a(x)}\geq 1$. Then, thanks to the continuity of the function $F(\lambda)$, $F(0)\geq 1$, and $\displaystyle \lim_{\lambda\to +\infty} F(\lambda)=0$, there exists $\lambda=\lambda_p\geq 0$ such that $(\lambda_\phi, \phi)$ is an eigenpair of (\ref{eq-11-3-2}), which proves $(b)$.
 \end{proof}
 
 A typical example for the function $a(x)$ is the function $a(x)=-\sqrt{x}$ defined on $\Omega=(0,1/5)$ satisfying that $\displaystyle\int_{0}^{1/5} \frac{dx}{-a(x)}=\int_{0}^{1/5} \frac{dx}{\sqrt{x}}=\frac{2}{\sqrt{5}}<1$. Let us fix $\epsilon\in \Big(0, \dfrac{\sqrt{5}}{2}-1\Big)$. We now consider the following eigenvalue problem:
 \begin{align}\label{eq-12-3-1}
	\left\{\begin{array}{ll}
	\displaystyle\int_{0}^{1/5} J(x-y) u(y)dy-u(x)+ \dfrac{2}{3}(1-\sqrt{x})u(x)+\dfrac{1}{3}(1-\sqrt{x})v(x)= \lambda_p u(x), & x \in\; [0, 1/5],  \\
	\displaystyle\int_{0}^{1/5} J(x-y) v(y)dy-v(x)+\dfrac{1}{3}(1-\sqrt{x})u(x)+ \dfrac{2}{3}(1-\sqrt{x})v(x) = \lambda_p v(x), &  x \in\; [0, 1/5],
	\end{array}\right.
	\end{align}
 where $J(x)\in \mathcal{C}_0^{\infty}\left((-1,1)\right)$ be a nonnegative, symmetric function such that $|J(x)-1|<\epsilon$ for all $x\in [0,1/5]$.
 \begin{proposition} 
We have that there exists no positive principal eigenfunction $(u,v)\in L^1([0,1/5])\times  L^1([0,1/5])$ to (\ref{eq-12-3-1}).
 \end{proposition}
 \begin{proof} We ague by contradiction. We suppose otherwise that there exists an eigenpair $((u,v), \lambda_p )$ of (\ref{eq-12-3-1}). 
 Then by adding two equations in (\ref{eq-12-3-1}) we obtain the following equation
\begin{equation}\label{eq-12-3-2}
 \displaystyle\int_{0}^{1/5} J(x-y) \phi(y)dy -\sqrt{x}\phi(x)= \lambda_p \phi(x), \; x\in [0,1/5],
\end{equation}
 where $\phi(x):=(u(x)+v(x))/2>0$. Without loss of generality, we may normalize that $ \displaystyle\int_{0}^{1/5}\phi(x)dx=1$. 
 Therefore, one has
$$
 \displaystyle \int_{0}^{1/5} \phi(y)dy+\int_{0}^{1/5} \left(J(x-y)-1\right) \phi(y)dy -\sqrt{x}\phi(x)= \lambda_p \phi(x), \; x\in [0,1/5].
$$
Since $\int_{0}^{1/5} \phi(y)dy=1$ and 
$$
\displaystyle \left|\int_{0}^{1/5} \left(J(x-y)-1\right) \phi(y)dy\right|\leq \int_{0}^{1/5} \left|J(x-y)-1\right| \phi(y)dy<\epsilon \int_{0}^{1/5} \phi(y)dy=\epsilon,\; \forall x\in [0,1/5],
$$
it follows that $\displaystyle 1-\epsilon\leq (\lambda_p+\sqrt{x}) \phi(x)\leq 1+\epsilon$ for all $x\in [0,1/5]$. This implies that $\lambda_p\geq 0$ and 
$$
\phi(x)\leq \frac{1+\epsilon}{\lambda_p+\sqrt{x}}, \; \forall x\in [0, 1/5],
$$
which yields 
$$
1=\int_{0}^{1/5} \phi(x)dx\leq \int_{0}^{1/5} \frac{1+\epsilon}{\lambda_p+\sqrt{x}}dx\leq \int_{0}^{1/5} \frac{1+\epsilon}{\sqrt{x}}dx=\frac{2}{\sqrt{5}} (1+\epsilon).
$$
This is a contradiction since $\displaystyle\frac{2}{\sqrt{5}} (1+\epsilon)<1$. Hence, the proof is complete.
 \end{proof}
 \begin{remark}
We note that the matrix function $A(x)$ is not Lipschitz continuous and its eigenvalue function $\lambda_1(x)=1-\sqrt{x}$ is continuous but not Lipschitz continuous. 
\end{remark}
 \subsection{The existence of the principal eigenvalue}
 In what follows, for $\delta>0$, let us recall 
 
 $$\Omega_\delta:=\{x\in \Omega\colon \|x-x_0\|<\delta\}.$$


\begin{proof}[\textbf{Proof of Theorem \ref{existence-eigen}}]
We shall following the ideas in \cite[Theorem $3.1$]{HMMV03}. However, much more elaborate computations are needed to deal with matrix-valued functions.

Firstly, for ${\varphi}=\left(\varphi_1,\ldots,\varphi_N\right)^T,\psi=\left(\psi_1,\ldots,\psi_N\right)^T\in \mathbf{E}$, we have
\begin{align*}
\left \langle \bm{\mathcal K}{\varphi}, \psi\right \rangle=&d_1\displaystyle\int_\Omega\displaystyle\int_\Omega J_1(x-y)\varphi_1(x)\psi_1(y)dxdy +\cdots+d_N\displaystyle\int_\Omega\displaystyle\int_\Omega J_N(x-y)\varphi_N(x)\psi_N(y)dxdy\\
&
+\displaystyle\int_\Omega \varphi(x)^T A(x) \psi(x) dx.\\
\left \langle{\varphi}, \bm{\mathcal K}\psi\right \rangle=&d_1\displaystyle\int_\Omega\displaystyle\int_\Omega J_1(x-y)\varphi_1(x)\psi_1(y)dxdy+\cdots +d_N\displaystyle\int_\Omega\displaystyle\int_\Omega J_N(x-y)\varphi_N(x)\psi_N(y)dxdy\\
&
+\displaystyle \int_\Omega \varphi(x)^T A(x) \psi(x) dx.
\end{align*}
Thanks to $A^T=A$, it is easy to see that $\left \langle \bm{\mathcal K}{\varphi}, \psi\right \rangle=\left \langle{\varphi}, \bm{\mathcal K}\psi\right \rangle$, which implies that $\bm{\mathcal K}$ is self-adjoint.

Now we show that  all eigenvalues of $\bm{\mathcal K}$ are real. Indeed, suppose that $\varphi$ is an eigenfuntion associated to an eigenvalue $\lambda$, i.e., $\bm{\mathcal K}\varphi=\lambda \varphi$. Then we also have $\bm{\mathcal K}\bar\varphi=\bar\lambda \bar\varphi$. Therefore, since $\langle \bm{\mathcal K}\varphi, \bar\varphi\rangle=\langle\varphi, \bm{\mathcal K}\bar\varphi\rangle$ it follows that 
$$
\lambda \langle\varphi, \bar\varphi\rangle=\bar \lambda \langle\bar \varphi,\varphi\rangle.
$$
This yields that $\lambda =\bar\lambda$, as desired. 

Let us define
\begin{align}\label{lamda+}
\lambda_- =- \sup\limits_{\left\|{\varphi}\right\|_{E}=1}\left \langle \bm{\mathcal K}{\varphi}, {\varphi}\right \rangle .
\end{align}

To complete the proof, we must first establish the following
\begin{align}\label{eq3.3}
\lambda_-<- \max_{x\in \overline{\Omega}}\bar\lambda(A(x)).
\end{align}

Set $\alpha=\max_{x\in \overline{\Omega}}\bar\lambda(A(x))$. Then, to prove (\ref{eq3.3}), it suffices to show that there exists $v\in \mathbf E$ such that 
$$
\alpha \|v\|_{\mathbf E}^2-\langle{\bm{\mathcal A}} v, v\rangle<\langle\bm{\mathcal N} v,v\rangle,
$$
that is
\begin{equation}\label{eq-15-3-1}
\displaystyle \int_\Omega v(x)^T\left(\alpha I-A(x)\right) v(x) dx <\sum_{j=1}^N\int_\Omega\int_\Omega  J_j(x-y) v_j(x) v_j(y) dxdy.
\end{equation}

 Consider the continuous eigenvalue functions $\lambda_1(x),\ldots,\lambda_N(x)$ of $A(x)$. Since $A(x)$ satisfies the condition $(\mathbf P)$, one can choose $x_0\in \overline{\Omega}$ such that $\max_{x\in \overline{\Omega}}\bar\lambda(A(x))=\bar\lambda(A(x_0))$ and there exists a continuous eigenpair function, say $(\lambda_1, e(x))$ with $\|e(x)\|=1$, of $A(x)$ such that the condition $(\mathbf P1)$ and $(\mathbf P2)$ are satisfied.  

Since $J_j(0)>0$ for all $j=1,\ldots, N$, there exists $\epsilon, \delta>0$ such that 
\begin{equation}\label{eq-15-3-2}
J_j(x-y) >\epsilon, \quad\forall  x,y\in  \Omega_\delta, \; j=1, \ldots, N.
\end{equation}
Moreover, since $e(x_0)\ne 0$ there exists $j_0\in\{1,\ldots, N\}$ such that $e_{j_0}(x_0)\ne 0$. Without loss of generality, we may assume that 
 $e_{j_0}(x_0)>0$. Thanks to the continuity of  the function $e_{j_0}$ at $x_0$,  by shrinking the domain $\Omega_\delta$ if necessary we may assume that
  \begin{equation}\label{eq-15-3-5}
 \begin{split}
 e_{j_0}(x)> \tau>0, \; \forall x\in \Omega_\delta 
 \end{split}
 \end{equation}
 for some constant $\tau>0$.

Now let us define $v(x):=g(x).e(x)$
\[\displaystyle 
g(x)=\begin{cases}
\dfrac{1}{\gamma+\lambda_1(x_0)-\lambda_1(x)}\quad &\text{if}\quad x\in \Omega_\delta\\
0 \quad &\text{if}\quad x\not\in \Omega_\delta,
\end{cases}
\]
where $\gamma>0$ will be chosen later. 

It follows from the condition ($\mathbf P2$) that there exists $\gamma>0$ small enough such that the following estimate holds true:
 \begin{equation}\label{eq-15-3-3}
 \begin{split}
 \displaystyle  \int_{\Omega _\delta}\dfrac{1}{\gamma+\lambda_1(x_0)-\lambda_1(x)} dx \geq \frac{1}{\epsilon \tau^2}.
 \end{split}
 \end{equation}
  Therefore, by (\ref{eq-15-3-2}) one gets
  \begin{equation}\label{eq-15-3-5}
 \begin{split}
 \sum_{j=1}^N\int_\Omega\int_\Omega  J_j(x-y) v_j(x) v_j(y) dxdy&= \sum_{j=1}^N\int_{\Omega_\delta}\int_{\Omega_\delta}  J_j(x-y) v_j(x) v_j(y) dxdy\\
 &\geq \epsilon \sum_{j=1}^N\int_{\Omega_\delta}\int_{\Omega_\delta}  v_j(x) v_j(y) dxdy\\
 &=\epsilon \sum_{j=1}^N\Big( \int_{\Omega_\delta} v_j(x) dx\Big)^2\\
 &=\epsilon \sum_{j=1}^N\Big( \int_{\Omega_\delta}\dfrac{1}{\gamma+\lambda_1(x_0)-\lambda_1(x)} e_j(x) dx\Big)^2\\
 &\geq\epsilon \Big( \int_{\Omega_\delta}\dfrac{1}{\gamma+\lambda_1(x_0)-\lambda_1(x)} e_{j_0}(x) dx\Big)^2\\
 &\geq\epsilon \tau^2\Big( \int_{\Omega_\delta}\dfrac{1}{\gamma+\lambda_1(x_0)-\lambda_1(x)} dx\Big)^2.
  \end{split}
 \end{equation}
 Furthermore, since $\displaystyle v(x)^T \left(A(x)\right) v(x)=\lambda_1(x) \|v(x)\|^2= g(x)^2\lambda_1(x)$ for all $x\in \Omega$, we have
 \begin{equation}\label{eq-15-3-4}
 \begin{split}
 \displaystyle  \int_\Omega v(x)^T \left(\alpha I-A(x)\right) v(x) dx&= \int_{\Omega _\delta}\dfrac{\lambda_1(x_0)-\lambda_1(x)}{\left(\gamma+\lambda_1(x_0)-\lambda_1(x)\right)^2} dx\\
 &\leq  \int_{\Omega _\delta}\dfrac{1}{\gamma+\lambda_1(x_0)-\lambda_1(x)} dx\\
  &\leq \epsilon C^2\Big( \int_{\Omega _\delta}\dfrac{1}{\gamma+\lambda_1(x_0)-\lambda_1(x)} dx\Big)^2\\
   &\leq \sum_{j=1}^N\int_\Omega\int_\Omega  J_j(x-y) v_j(x) v_j(y) dxdy,
 \end{split}
 \end{equation}
 where the last inequality follows from (\ref{eq-15-3-5}). This yields (\ref{eq-15-3-1}) and $(\ref{eq3.3})$ hence follows.

From the equation $\eqref{lamda+}$, it is standard that there is a sequence $\left\{{\varphi}^k \right\}\subset \mathbf{E}$ with $\left\|{\varphi}^k\right\|_{\mathbf{E}} = 1, (\forall k\in \mathbb{N})$ such that
\begin{align*}
\lim\limits_{k\to+\infty}\left\|\left(\bm{\mathcal K} + \lambda_-\bm{\mathcal I}\right){\varphi}^k \right\|_{\mathbf{E}}=0.
\end{align*}
Let us define the operator: 
$\widetilde{\bm{\mathcal{H}}}: \mathbf{E}\to \mathbf{E}$ satisfies
$\left(\widetilde{\bm{\mathcal{H}}}{\varphi}\right)(x) = \left[-\lambda_- \bm{\mathcal I}- \bm{\mathcal A} (x)\right]{\varphi}(x).$
It follows by \eqref{eq3.3}, we deduce that $\det(-\lambda_- I-A(x))\neq 0$ (since $\displaystyle -\lambda_- >\sup_{x\in \overline{\Omega}}\bar\lambda(A(x))$), so $\widetilde{\bm{\mathcal{H}}}$ has a bounded inverse. It is notice that
\begin{align*}
\left(\bm{\mathcal K} + \lambda_-\bm{\mathcal I}\right){\varphi}^k = \bm{\mathcal N}{\varphi}^k - \widetilde{\bm{\mathcal{H}}}{\varphi}^k, \,\,\,\text{and}\,\,\, \widetilde{\bm{\mathcal{H}}}^{-1}\bm{\mathcal N}{\varphi}^k - {\varphi}^k = \widetilde{\bm{\mathcal{H}}}^{-1}\left(\bm{\mathcal N}{\varphi}^k - \widetilde{\bm{\mathcal{H}}}{\varphi}^k\right).
\end{align*}
 Since $\bm{\mathcal N}$ is compact, there is a subsequence, still denoted by $\left\{{\varphi}^k\right\}$ such that $\bm{\mathcal N}{\varphi}^k\to {v}\in \mathbf{C}$; let ${\phi} = \widetilde{\bm{\mathcal{H}}}^{-1}{ v}\in\mathbf{C}$. (We emphasize here that although the sequence $\left\{{\varphi}^k \right\}\subset \mathbf{E}$, but the sequence of $\bm{\mathcal N}{\varphi}^k \in \mathbf{C}$ and thus ${ v}\in \mathbf C$.) Then $\lim\limits_{k\to+\infty}\widetilde{\bm{\mathcal{H}}}^{-1}\bm{\mathcal N}{\varphi}^k = \widetilde{\bm{\mathcal{H}}}^{-1}{v}={\phi}$. This implies $\widetilde{\bm{\mathcal{H}}}^{-1}\bm{\mathcal N}{\varphi}^k - {\varphi}^k \to 0,\,\,{\varphi}^k \to {\phi}$. We deduce that $\widetilde{\bm{\mathcal{H}}}^{-1}\bm{\mathcal N}{\phi} = {\phi}$, which leads to 
 \begin{align}\label{eigenfunction}
 \bm{\mathcal K}{\phi} +\lambda_-{\phi}=0.
 \end{align}
 Note that ${\phi}\neq 0$ since $\|{\varphi}^k\|_{\mathbf{E}}=1$ and ${\varphi}^k\rightarrow {\phi}$. Therefore  $\psi=\left(\phi_1,\ldots,\phi_N\right)$  (assumed normalised) is an eigenfunction of $\bm{\mathcal K}$ corresponding to the eigenvalue $\lambda_p=\lambda_-$, and $\phi$ is continuous. Note that,  with test function $\psi=\left(|\phi_1|,\ldots,|\phi_N| \right)$, we have
 \begin{align*}
 \begin{array}{lll}
  \left \langle \bm{\mathcal K}\psi,\psi\right \rangle+\lambda_-&= \left \langle \bm{\mathcal K}\psi,\psi\right \rangle-\left \langle \bm{\mathcal K}{\phi},{\phi}\right \rangle\\
  &= \displaystyle \sum_{j=1}^N d_j\displaystyle\int_\Omega\displaystyle\int_\Omega J_j(x-y)\left(|\phi_j(x)\phi_j(y)|-\phi_j(x)\phi_j(y)\right)dxdy\\
&+\displaystyle 2\sum_{i\ne j}\int_\Omega a_{ij}(x) (\left(|\varphi_i(x)\phi_j(x)|-\phi_i(x)\phi_j(x)\right)dx  \geq 0.
 \end{array}
 \end{align*} 
 Hence, by $\eqref{lamda+}$ we induce that
 $\displaystyle\int_\Omega\displaystyle\int_\Omega J_i(x-y)\left(|\phi_i(x)\phi_i(y)|-\phi_i(x)\phi_i(y)\right)dxdy=0, i=1,\ldots, N,\;\text{and}$
  $\displaystyle\int_\Omega \left(|\phi_i(x)\phi_j(x)|-\phi_i(x)\phi_j(x)\right)dx=0$ for all $i\ne j$ (note that $a_{ij}(x)> 0$ for any $1\leq i,j\leq N$).
 Since the functions under the integrals are continuous and non-negative, we follow that
 \begin{align}\label{3.10}
 \phi_i(x)\phi_i(y)\geq 0,\text{ for all }x,y\in \Omega, 1\leq i\leq N.
 \end{align}

Now, if there exists $x_0\in \Omega$ such that $\phi_i(x_0)<0$, then from $\eqref{3.10}$,
 we have $-{\phi}\in \mathbf{E}^{+}$. In this case, we can choose $-{\phi}$, which is non-negative eigenfunction. Therefore, we may assume that $\phi\geq 0$. Note that $\phi$ cannot be a eigenfunction if $\phi_1(x_0)=0$ or $\phi_2(x_0)=0$ for some $x_0\in \overline{\Omega}$. For if it were, assume $\phi_1(x_0)=0$, putting $x=x_0$ in equation $\eqref{eigenfunction}$ gives
\begin{align}\label{x0}
\left\{\begin{array}{lll}
d_1\displaystyle\int_\Omega J_1(x_0-y)\phi_1(y)dy+\sum_{j=2}^N a_{1j}\phi_j(x_0)&=0,\\
d_2\displaystyle\int_\Omega J_2(x_0-y)\phi_2(y)dy+(a_{22}-d_2)\phi_2(x_0)+\sum_{j=1}^N a_{2j}\phi_j(x_0)&=-\lambda_-\phi_2(x_0)\\
\vdots&\vdots\\
d_N\displaystyle\int_\Omega J_N(x_0-y)\phi_N(y)dy+(a_{NN}-d_N)\phi_N(x_0)+\sum_{j=1}^{N} a_{jN}\phi_j(x_0)&=-\lambda_-\phi_N(x_0)\\
\end{array}\right.
\end{align}
Due to ${\bf {(J)}}$ and $\phi_j\geq 0$ then the first equation of $\eqref{x0}$ yields $\phi_1(x)=0,\forall x\in \Omega$ and $\phi_j(x_0)=0$ for all $2\leq j\leq N$. Now, the equations of $\eqref{x0}$ lead to $\phi_j(x)=0,\forall x\in \Omega$ for all $2\leq j\leq N$ . This contradiction proves that ${\phi}>0$. The uniqueness can be obtained by a simple consequence, for if ${\phi}$ and $\psi$ were different eigenfunctions, ${\phi}-\psi$ would be an eigenfunction. But this may change sign, contradicting the positive. In orther to show the conclusion, it is sufficient to prove that $\lambda\leq \lambda_-$ for any eigenvalue $\lambda$. Assume by contradiction that $\lambda>\lambda_-$ we have $\left \langle \bm{\mathcal K} {\varphi}, {\varphi} \right \rangle \leq \lambda_-\|{\varphi}\|^{2}, \; \forall {\varphi} \in E,$
and thus
\begin{align*}
\left \langle \lambda {\varphi}- \bm{\mathcal K} {\varphi}, {\varphi}\right \rangle  \geq\left(\lambda-\lambda_{-}\right)\|{\varphi}\|^{2}=\alpha\|{\varphi}\|^{2}, \forall {\varphi} \in \mathbf{E} \;\text{ with } \alpha>0.
\end{align*}
Applying Lax–Milgram’s theorem, we deduce that $\lambda \bm{\mathcal I}-\bm{\mathcal K}$ is bijective
and thus $\lambda\in \rho(\bm{\mathcal K})$. This is impossible as $\lambda$ is eigenvalue of $\bm{\mathcal K}$.

\end{proof}

 \begin{remark}
 \begin{itemize}
  
\item[i)]We note that the matrix function $A(x)$ is  not Lipschitz continuous and its eigenvalue function $\lambda_1(x)=1-\sqrt{x}$ is continuous but not Lipschitz continuous. Therefore, Theorem \ref{existence-eigen} cannot be applied to the system (\ref{eq-12-3-1}). In fact, there is no positive principal eigenfunction $(u,v)\in L^1([0,1/5])\times  L^1([0,1/5])$ to (\ref{eq-12-3-1}). 
\item[ii)] Following the proof of Theorem \ref{existence-eigen}, one sees that Theorem \ref{existence-eigen} holds for the case that the kernel function $J_i(x-y) \in \mathcal{C}^0(\overline{\Omega})$ is non-negative and $J_i(x-x)=J_i(0)>0$ for all $1\leq i\leq N$ and $x\in\overline{\Omega}$. 
\end{itemize}
\end{remark}
 
\begin{proposition}\label{proposition_3.4}
	Assume ${\bf(J)}$ hold and let $A(x) = (a_{ij}(x))\in  (\mathcal{C}(\overline{\Omega}))^{N\times N}, B(x) = (b_{ij}(x))\in  (\mathcal{C}(\overline{\Omega}))^{N\times N}$ be two symmetric matrix-valued functions on $\overline{\Omega}\subset\mathbb R^n$.
\begin{itemize}	
\item[a)] If there exists a function $\widetilde{{\varphi}}=\left(\widetilde{\varphi}_1,\ldots, \widetilde{\varphi}_N\right)\in \mathbf{C}\left(\overline{\Omega}\right)$ with $\widetilde{\varphi}_1,\ldots, \widetilde{\varphi}_N\geq, \not\equiv 0$ in $\overline{\Omega}$ and a constant $\widetilde{\lambda}$ such that
\begin{align}\label{EPV}
\left({\bf \bm{\mathcal D} \bm{\mathcal N}} - \bm{\mathcal D} +\bm{\mathcal A}\right)[\widetilde{{\varphi}}]+\widetilde{\lambda}\widetilde{{\varphi}}\leq 0,\;\;\text{in}\;\;\Omega,
\end{align}
then $\lambda_p\geq \widetilde{\lambda}$, where $\lambda_p$ is the eigenvalue of problem \eqref{PEV}. Moreover, $\lambda_p= \widetilde{\lambda}$ only if equality holds in \eqref{EPV}.

\item[b)] If there exists a function $\widetilde{{\varphi}}=\left(\widetilde{\varphi}_1, \ldots,\widetilde{\varphi}_N\right)\in  \mathbf{C}\left(\overline{\Omega}\right)$ with $\widetilde{\varphi}_1, \ldots, \widetilde{\varphi}_N\geq, \not\equiv 0$ in $\overline{\Omega}$ and a constant $\widetilde{\lambda}$ such that
\begin{align}\label{EPV2}
\left({\bf \bm{\mathcal D} \bm{\mathcal N}} - \bm{\mathcal D} +\bm{\mathcal A}\right)[\widetilde{{\varphi}}]+\widetilde{\lambda}\widetilde{{\varphi}}\geq 0,\;\;\text{in}\;\; \Omega,
\end{align}
then $\lambda_p\leq \widetilde{\lambda}$, where $\lambda_p$ is the eigenvalue of problem \eqref{PEV}. Moreover, $\lambda_p= \widetilde{\lambda}$ only if equality holds in \eqref{EPV2}.

\item[c)] If $a_{ij}(x)\leq b_{ij}(x)$ for all $i,j=1,\ldots, N$ and for all $x\in \Omega$ then $\lambda_p(\bm{\mathcal A})+m\leq \lambda_p(\bm{\mathcal B})$, where $m=\inf_{\Omega}\min\{a_{11}(x)-b_{11}(x),\ldots,a_{nn}(x)-b_{nn}(x)\}$.
 \end{itemize}	
\end{proposition}
\begin{proof}[\bf Proof]
First, we just prove the conclusion a) since the proof of b) is similar. Let ${\varphi}=\left(\varphi_1,\ldots, \varphi_N\right)$ be positive eigenfunction corresponding to the principal eigenvalue $\lambda_p(\bm{\mathcal K})$ in \eqref{PEV}, where $\bm{\mathcal K}:= {\bf \bm{\mathcal D} \bm{\mathcal N}} - \bm{\mathcal D} +\bm{\mathcal A}$. We have
\begin{align*}
\left\langle\bm{\mathcal K}\widetilde{{\varphi}},{\varphi}\right\rangle&=\displaystyle\int_\Omega\left(d_1\displaystyle\int_\Omega J_1(x-y)\widetilde{\varphi}_1(y)-d_1\widetilde{\varphi}_1(x)+\sum_{j=1}^N a_{1j}(x)\widetilde{\varphi}_j(x)\right)\varphi_1(x)dx+\cdots\\
&+\displaystyle\int_\Omega\left(d_N\displaystyle\int_\Omega J_N(x-y)\widetilde{\varphi}_N(y)-d_N\widetilde{\varphi}_N(x)+\sum_{j=1}^N a_{nj}(x)\widetilde{\varphi}_j(x)\right)\varphi_N(x)dx\\
&\leq -\widetilde{\lambda}\int_\Omega \left\langle\widetilde{{\varphi}}(x),{\varphi}(x)\right\rangle dx= -\tilde\lambda \left \langle\widetilde{{\varphi}},{\varphi}\right\rangle,
\end{align*}
which implies that
$\left\langle\bm{\mathcal K}\widetilde{{\varphi}},{\varphi}\right\rangle\leq -\widetilde{\lambda}\left\langle\widetilde{{\varphi}},{\varphi}\right\rangle .$
Due to the definition of ${\varphi}$ we also have
$\left\langle\bm{\mathcal K}{\varphi},\widetilde{{\varphi}}\right\rangle= -\lambda_p(\bm{\mathcal K})\left \langle\widetilde{{\varphi}},{\varphi}\right\rangle.$
Since $\bm{\mathcal K}$ is self-adjiont and $\left\langle\widetilde{{\varphi}}, {\varphi}\right\rangle>0$, we obtain that $\lambda_p(\bm{\mathcal K})\geq \widetilde{\lambda}$. Moreover, if one of the inequalities in \eqref{PEV} is strict at some point $x_0\in \overline{\Omega}$, then $\lambda_p(\bm{\mathcal K})> \widetilde{\lambda}$.

Next, we prove the conclusion c).	 Let $\left(\varphi_1^{\bm{\mathcal A}},\ldots, \varphi_N^{\bm{\mathcal A}}\right)$ and $\left(\varphi_1^{\bm{\mathcal B}},\ldots,\varphi_N^{\bm{\mathcal B}}\right)$   be the  corresponding eigenfunction to $\lambda_p\left(\bm{\mathcal A}\right)$ and $\lambda_p\left(\bm{\mathcal B}\right)$, respectively. Then we have
	\begin{align*}
	&d_1\displaystyle\int_\Omega J_1(x-y)\varphi_1^{\bm{\mathcal A}}(y)dy-d_1\varphi_1^{\bm{\mathcal A}}(x)+\sum_{j=1}^N b_{1j}(x)\varphi_j^{\bm{\mathcal A}}(x)+\left(\lambda_p(\bm{\mathcal A})+m\right)\varphi_1^{\bm{\mathcal A}}\\
	\leq& d_1\displaystyle\int_\Omega J_1(x-y)\varphi_1^{\bm{\mathcal A}}(y)dy-d_1\varphi_1^{\bm{\mathcal A}}(x)+(b_{11}(x)+m)\varphi_1^{\bm{\mathcal A}}(x)+\sum_{j=2}^N a_{1j}(x)\varphi_j^{\bm{\mathcal A}}(x)+\lambda_p(\bm{\mathcal A})\varphi_1^{\bm{\mathcal A}}\\
	\leq& d_1\displaystyle\int_\Omega J_1(x-y)\varphi_1^{\bm{\mathcal A}}(y)dy-d_1\varphi_1^{\bm{\mathcal A}}(x)+\sum_{j=1}^N a_{1j}(x)\varphi_2^{\bm{\mathcal A}}(x)+\lambda_p(\bm{\mathcal A})\varphi_1^{\bm{\mathcal A}},
	\end{align*}
	which implies that
	\begin{align*}
	d_1\displaystyle\int_\Omega J_1(x-y)\varphi_1^{\bm{\mathcal A}}(y)dy-d_1\varphi_1^{\bm{\mathcal A}}(x)+\sum_{j=1}^N b_{1j}(x)\varphi_j^{\bm{\mathcal A}}(x)+\left(\lambda_p(\bm{\mathcal A})+m\right)\varphi_1^{\bm{\mathcal A}}\leq 0.
	\end{align*}
	Similarly, we also obtain
	\begin{align*}
	d_i\displaystyle\int_\Omega J_i(x-y)\varphi_i^{\bm{\mathcal A}}(y)dy-d_i\varphi_i^{\bm{\mathcal A}}(x)+\sum_{j=1}^N b_{ij}(x)\varphi_j^{\bm{\mathcal A}}(x)+\left(\lambda_p(\bm{\mathcal A})+m\right)\varphi_i^{\bm{\mathcal A}}\leq 0
	\end{align*}
for all $i=1,\ldots, N$.
	Using the conclusion a), we get $\lambda_p(\bm{\mathcal A})+m\leq \lambda_p(\bm{\mathcal B})$.
\end{proof}

\begin{lemma}\label{lem3.5-5} We have
$$		\lambda_p( \bm{\mathcal K})=\lambda_p'( \bm{\mathcal K})=- \sup\limits_{\left\|{\varphi}\right\|_{E}=1}\left \langle \bm{\mathcal K}{\varphi}, {\varphi}\right \rangle=\lambda_v( \bm{\mathcal K})=\lambda_1( \bm{\mathcal K}).
$$
\end{lemma}
\begin{proof}
	
By Theorem \ref{existence-eigen} and Proposition \ref{proposition_3.4}, one has
$$		\lambda_p'( \bm{\mathcal K})\leq\lambda_p( \bm{\mathcal K})=- \sup\limits_{\left\|{\varphi}\right\|_{E}=1}\left \langle \bm{\mathcal K}{\varphi}, {\varphi}\right \rangle=\lambda_v( \bm{\mathcal K})=\lambda_1( \bm{\mathcal K}).
$$	
We shall prove that $\lambda_v( \bm{\mathcal K})\leq\lambda_p'( \bm{\mathcal K})$. 
Indeed,
let  $\lambda >\lambda_p'( \bm{\mathcal K})$, then by definition of $\lambda_p'$ there exists $\psi\geq 0$ such that $\psi \in \mathbf C(\Omega)\cap L^\infty(\Omega)$ and 
\begin{equation}\label{bcv-eq-l2sup}
	 \bm{\mathcal K}[\psi](x)+\lambda \psi(x)\geq 0\quad \text{ in }\quad \O.
\end{equation}
Since $\O$ is bounded and $\psi\in \mathbf C(\overline{\Omega})$,    $\psi \in E:= \left(L^2(\Omega\right))^N$. So, by taking the scalar product \eqref{bcv-eq-l2sup} with $\psi$ and integrating over $\O$ we get 
\begin{align*}
	 \int_{\Omega} \langle\bm{\mathcal K}[\psi](x),\psi(x)\rangle dx &\geq -\lambda\int_{\Omega} \langle \psi(x),\psi(x)\rangle dx\\
	\Leftrightarrow	\langle\bm{\mathcal K}[\psi],\psi\rangle_{E}  &\geq -\lambda\langle \psi,\psi\rangle_{E} \\
	\Leftrightarrow \dfrac{	\langle\bm{\mathcal K}[\psi],\psi\rangle_{E} }{ \|\psi\|^2_{E}}&\geq  -\lambda.
\end{align*}
This yields that 
$$
\lambda_v(\bm{\mathcal K})=-\sup_{0\not \equiv\psi\in \mathbf C(\Omega)\cap L^\infty(\Omega)} \dfrac{	\langle\bm{\mathcal K}[\psi],\psi\rangle_{E} }{ \|\psi\|^2_{E}}\leq   \lambda 
$$
for all  $\lambda >\lambda_p'( \bm{\mathcal K})$. Therefore, $\lambda_v(\bm{\mathcal K})\leq \lambda_p'(\bm{\mathcal K})$. Thus, the proof is complete.
\end{proof}

We now establish a maximum principle for the operator $\bm{\mathcal K}$ defined in \eqref{PEV}, which will be useful for our later analysis.
\begin{proposition}[Maximum principle]\label{mpnt}
 Assume \textbf{(J)} holds and let $A(x)\in  (\mathcal{C}(\overline{\Omega}))^{N\times N}$ be a continuous symmetric matrix-valued function on $\overline{\Omega}$ such that $a_{ij}(x)=a_{ji}(x)>0$ whenever $i\ne j$. If $\lambda_p(\bm{\mathcal K})\geq 0$, then for all function ${\varphi}=\left (\varphi_1,\ldots,\varphi_N\right)\in  \mathbf{C}\left(\overline{\Omega}\right)$ satisfying
 \begin{align*}
 \bm{\mathcal K}{\varphi}\leq 0,\qquad &\text{ in } \Omega,\\
 {\varphi} \geq 0, \qquad &x\in \partial \Omega,
 \end{align*}
 then we have ${\varphi}\geq 0$ in $\Omega$.
\end{proposition}
\begin{proof}[\bf Proof]
Let ${\varphi}=\left (\varphi_1,\ldots,\varphi_N\right)\in \mathbf{C}\left(\overline{\Omega}\right)$, ${\varphi}\not \equiv 0$ satisfies $\varphi_i(x)\geq 0$ with $i=1,\ldots, N$ and $x\in \partial \Omega$ and
\begin{align}\label{mp3}
d_i\displaystyle\int_\Omega J_i(x-y)\varphi_i(y)dy-d_i\varphi_i(x)+\sum_{j=1}^N a_{ij}(x)\varphi_j(x)\leq 0,\; \forall x\in \Omega, \; i=1,\ldots, N.
\end{align}

	Let ${\varphi}^p=(\varphi_1^p,\ldots,\varphi_N^p)\in  \mathbf{C}\left(\overline{\Omega}\right)$ be the corresponding eigenfunction to $\lambda_p(\bm{\mathcal K})$. Then we have $\varphi^p>0$ and
	\begin{align}\label{mp8}
 d_i\displaystyle\int_\Omega J_i(x-y)\varphi_i^p(y)dy-d_i\varphi_i^p(x)+\sum_{j=1}^N a_{ij}(x)\varphi_j^p(x)+\lambda_p(\bm{\mathcal K}))\varphi_i^p(x)=0,\; \forall x\in \Omega, \; i=1,\ldots, N.
	\end{align}
	
	 Let us define
	\begin{align*}
	\psi=(\psi_1,\ldots,\psi_N):=\left(\dfrac{\varphi_1}{\varphi_1^p},\ldots,\dfrac{\varphi_N}{\varphi_N^p}\right).
	\end{align*}
Combining \eqref{mp3} and \eqref{mp8}, we have
	\begin{align}\label{mp1}
	&d_1\int_\Omega J_1(x-y)\varphi_1^p(y)(\psi_1(y)-\psi_1(x))dy-\lambda_p(\bm{\mathcal K})\varphi_1^p(x)\psi_1(x)\nonumber\\
	=&d_1\int_\Omega J_1(x-y)\left(\varphi_1(y)-\varphi_1^p(y)\dfrac{\varphi_1(x)}{\varphi_1^p(x)}\right)dy-\lambda_p(\bm{\mathcal K})\varphi_1^p(x)\dfrac{\varphi_1(x)}{\varphi_1^p(x)}\nonumber\\
	=&d_1\int_\Omega J_1(x-y)\varphi_1(y)dy-\dfrac{\varphi_1(x)}{\varphi_1^p(x)}\left(d_1\int_\Omega J_1(x-y)\varphi_1^p(y)dy+\lambda_p(\bm{\mathcal K})\varphi_1^p(x)\right)\nonumber\\
	=& d_1\int_\Omega J_1(x-y)\varphi_1(y)dy+\dfrac{\varphi_1(x)}{\varphi_1^p(x)}\left[\sum_{j=1}^N a_{1j}(x)\varphi_j^p(x)-d_1\varphi_1^p(x)\right]\nonumber\\
	=& d_1\int_\Omega J_1(x-y)\varphi_1(y)dy-d_1\varphi_1(x)+a_{11}(x)\varphi_1(x)+\sum_{j=2}^N a_{1j}(x)\varphi_1(x)\dfrac{\varphi_j^p(x)}{\varphi_1^p(x)}\nonumber\\
	\leq&-\sum_{j=2}^N a_{1j}(x)\varphi_j(x)+\sum_{j=2}^N a_{1j}(x)\varphi_1(x)\dfrac{\varphi_j^p(x)}{\varphi_1^p(x)}
	=\sum_{j=2}^N a_{1j}(x) \varphi_j^p(x)\left(\psi_1(x)-\psi_j(x)\right), \;\text{for any}\; x\in \Omega.
	\end{align}
	Similarly,  \eqref{mp3} and \eqref{mp8}, we also obtain that
	\begin{align}\label{mp2}
	d_2\int_\Omega J_2(x-y)\varphi_2^p(y)\left(\psi_2(y)-\psi_2(x)\right)dy-\lambda_p(\bm{\mathcal K})\varphi_2^p(x)\psi_2(x)
	\leq \sum_{j\ne 2}a_{2j}(x)\varphi_j^p(x)[\psi_2(x)-\psi_j(x)].
	\end{align}
Since $\psi_1,\ldots,\psi_N$ are  continuous functions in $\overline{\Omega}$, $\psi_1,\ldots,\psi_N$ achieve at some $x_1,\ldots,x_N\in \overline{\Omega}$ a minimum, respectively, i.e $\psi_j(x_j)=\min_{\overline{\Omega}}\psi_j(x)$ for $j=1,\ldots, N$. Without loss of generality, we may assume that $\psi_N(x_N)\leq \cdots\leq \psi_1(x_1)$. We first prove that $\psi_1(x_1)\geq 0$. If it is not true (that is, $\psi_j(x_j)<0$ for $j=1,\ldots, N$) and due to $\varphi_j(x)\geq 0$ for all $x\in \partial \Omega$, then $\psi_j(x_j)<0$ with $x_j\in \Omega$.

By putting $x=x_1$ into \eqref{mp1} we have
	\begin{align}\label{eqt-1-11-1}
	0&\leq \int_\Omega J_1(x_1-y)\varphi_1^p(y)\left(\psi_1(y)-\psi_1(x_1)\right)dy-\lambda_p(\bm{\mathcal K})\varphi_1^p(x_1)\psi_1(x_1)\\
	&\leq \sum_{j=2}^N a_{1j}(x_1)\varphi_j^p(x_1)\left[\psi_1(x_1)-\psi_j(x_1)\right]\leq  \sum_{j=2}^N a_{1j}(x_1)\varphi_j^p(x_1)\left[\psi_1(x_1)-\psi_j(x_j)\right] \nonumber.
	\end{align}
Then putting $x=x_2$ into \eqref{mp2} we arrive at
	\begin{align}\label{eq-7-1}
		0&\leq\int_\Omega J_2(x_2-y)\varphi_2^p(y)\left(\psi_2(y)-\psi_2(x_2)\right)dy-\lambda_p(\bm{\mathcal K})\varphi_2^p(x_2)\psi_2(x_2)\\
	&\leq \sum_{j\ne 2} a_{2j}(x_2)\varphi_j^p(x_2)\left[\psi_2(x_2)-\psi_j(x_2)\right]\leq \sum_{j\ne 2} a_{2j}(x_2)\varphi_j^p(x_2)\left[\psi_2(x_2)-\psi_j(x_j)\right].\nonumber
	\end{align}
By induction, putting $x=x_N$ into \eqref{mp2} one obtains that 
	\begin{align}\label{eq-7-1-1}
	0&\leq\int_\Omega J_N(x_N-y)\varphi_N^p(y)\left(\psi_N(y)-\psi_N(x_N)\right)dy-\lambda_p(\bm{\mathcal K})\varphi_N^p(x)\psi_N(x_2)\\
	&\leq \sum_{j\ne N} a_{Nj}(x_N)\varphi_j^p(x_N)\left[\psi_N(x_N)-\psi_j(x_N)\right]\leq \sum_{j\ne N} a_{Nj}(x_N)\varphi_j^p(x_N)\left[\psi_N(x_N)-\psi_j( x_j)\right]\leq 0.\nonumber
\end{align}

Since $a_{Nj}(x_N)>0, \varphi_j^p(x_N)>0$ and $\psi_N(x_N)-\psi_j( x_j)\leq 0$ for $j=1,\ldots,N-1$, it follows from (\ref{eq-7-1-1}) that 
$$
\psi_1(x_1)=\psi_2(x_2)=\cdots=\psi_N(x_N)=\beta<0.
$$ 
Moreover, since $a_{1j}(x_1)>0, \varphi_j^p(x_1)>0$ for $j=2,\ldots,N$, (\ref{eqt-1-11-1}) implies that $\psi_j(x_1)=\beta$ for all $j=1,\ldots,N$. By induction, one concludes that $\psi_i(x_j)=\beta$  for all $i,j=1,\ldots, N$. Furthermore, it follows from (\ref{eqt-1-11-1}), (\ref{eq-7-1}), and (\ref{eq-7-1-1}) that
	\begin{align*}
 \int_\Omega J_i(x_i-y)\varphi_i^p(y)\left(\psi_i(y)-\psi_i(x_i)\right)dy-\lambda_p(\bm{\mathcal K})\varphi_i^p(x_i)\psi_i(x_i)=0, \; i=1,\ldots, N.
\end{align*}
Thanks to the assumption that $\lambda_p(\bm{\mathcal K})\geq 0$ and note that $\psi_i(x_i)=\beta<0$ for all $i=1,\ldots,N$, these equations yield that $\lambda_p(\bm{\mathcal K})=0$ and  
	\begin{align*}
	\int_\Omega J_i(x_i-y)\varphi_i^p(y)\left(\psi_i(y)-\psi_i(x_i)\right)dy=0, \; i=1,\ldots, N.
\end{align*}

	Since $\varphi_i^p,\psi_i$ and $J_i$ are continuous and non-negative functions for all $1\leq i\leq N$, the above inequality leads to
	\begin{align}\label{mp5}
	J_i(x_i-y)(\psi_i(y)-\psi_i(x_i))=0, \text { for all } y\in \Omega,\ i=1,\ldots, N.
	\end{align}
    By \textbf{(J)}, there exist constants $\delta_i>0$ such that
    \begin{align*}
    J_i(x_i-y)>0, \text{ for all } y\in B(x_i,\delta_i).
    \end{align*}
    By \eqref{mp5}, one easily sees that
    \begin{align}\label{mpe}
    \psi_i(y)=\psi_i(x_i), \text{ for all } y\in B(x_i,\delta_i), \; i=1,\ldots, N.
    \end{align}
    So $\psi_i$ also achieve negative minimum at any $x\in B[x_i,\delta_i/2]$. By repeating the above argument with replacing $x_i$ by $x\in \partial B(x_i,\delta_i/2)$, we can extend \eqref{mpe} to
    $\psi_i(y)=\psi_i(x_i), \text{ for all } y\in B(x_i,\frac{3}{2}\delta_i).$
    	Since $\delta_i$ does not change after each iteration, by repeating this process finitely many times we induce that $\psi_i(x)$ is constant in $\overline{\Omega}$ ($\psi_i(x)=\alpha<0$) which leads to $\varphi_i\equiv \alpha<0$ in $\overline{\Omega}$, which is impossible. Therefore, we conclude that $\psi_1(x_1)\geq 0$, and hence $\psi_1(x)\geq 0$. 
    	
    	Next, we prove that $\psi_2(x_2)\geq 0$ (note that one already has $\psi_1(x_1)\geq 0$). Suppose otherwise that $\psi_2(x_2)<0$. Namely,
    	$$
    	\psi_N(x_N)\leq \cdots\leq \psi_2(x_2)<0\leq \psi_1(x_1).
    	$$
    	Since $\psi_2(x_2)\leq \psi_1(x_1)$, the equation (\ref{eq-7-1}) becomes
    		\begin{align}\label{eq-7-2}
    		0&\leq\int_\Omega J_2(x_2-y)\varphi_2^p(y)\left(\psi_2(y)-\psi_2(x_2)\right)dy-\lambda_p(\bm{\mathcal K})\varphi_2^p(x_2)\psi_2(x_2)\\
    		&\leq \sum_{j\ne 2} a_{2j}(x_2)\varphi_j^p(x_2)\left[\psi_2(x_2)-\psi_j(x_2)\right]\leq  \sum_{j\ne 2} a_{2j}(x_2)\varphi_j^p(x_2) \left[\psi_2(x_2)-\psi_j(x_j)\right]\nonumber\\
    		&\leq  \sum_{j=3}^N a_{2j}(x_2)\varphi_j^p(x_2) \left[\psi_2(x_2)-\psi_j(x_j)\right].\nonumber
    	\end{align}
 By induction, since $\psi_N(x_N)\leq \psi_1(x_1)$, the equation (\ref{eq-7-1-1}) becomes
 \begin{align*}
 	0&\leq\int_\Omega J_N(x_N-y)\varphi_N^p(y)\left(\psi_N(y)-\psi_N(x_N)\right)dy-\lambda_p(\bm{\mathcal K})\varphi_N^p(x)\psi_N(x_2)\\
 	&\leq \sum_{j\ne N} a_{Nj}(x_N)\varphi_j^p(x_N)\left[\psi_N(x_N)-\psi_j(x_N)\right]\leq  \sum_{j\ne N} a_{Nj}(x_N)\varphi_j^p(x_N) \left[\psi_N(x_N)-\psi_j(x_j)\right]\\
 	&\leq   \sum_{j\ne 1, N} a_{Nj}(x_N)\varphi_j^p(x_N) \left[\psi_N(x_N)-\psi_j(x_j)\right]\leq 0.
 \end{align*}   	
    	
By a similar argument as above for $\psi_2,\ldots, \psi_N$ one derives a contradiction that $\psi_2(x)\equiv \alpha<0$. Therefore, $\psi_2(x_2)\geq 0$.
We repeat the argument above, by induction one conclude that $\psi_i(x)\geq 0$ for all $x\in \Omega$ and for all $i=1,\ldots, N$. Therefore, the proof is finally complete.
\end{proof}

\section{Asymptotic behaviors with respect to  dispersal rate }\label{sec.2}

Next, we investigate the effects of the dispersal rate characterized by $\bm{\mathcal D}$ on the principal spectrum point. In this section, let us denote
$$
\lambda_p(\bm{\mathcal D}):=\lambda_p(\bm{\mathcal D}\bm{\mathcal N}-\bm{\mathcal D}+ \bm{\mathcal A})),
$$
where  $\bm{\mathcal N}\colon \mathbf{E}\rightarrow \mathbf{C}$ given by
	\begin{align*}
	\left(\bm{\mathcal N}\varphi\right)(x)=\mathrm{diag}\left({\mathcal N}_1[\varphi_1](x),\ldots,\; {\mathcal N}_N[\varphi_N](x)\right),
	\end{align*}
	and ${\mathcal N}_i[\varphi_i](x): = 
	\displaystyle\int_\Omega J_i(x-y)\varphi_i(y)dy, \;\text{for} \; \varphi\in \mathbf{E}, i=1,\ldots, N$. 

We can prove the result Theorem \ref{pro.3.6}

\begin{proof}[\bf \textit{Proof of Theorem \ref{pro.3.6}}]
	
	\noindent (1) 	Let ${\varphi}^p=\left(\varphi^p_1,\ldots, \varphi^p_N\right)$ be the corresponding eigenfunction, which is normalized $\left\|{\varphi}^p\right\|_{\mathbf E}=1$, associated to the eigenvalue $\lambda_p(\bm{\mathcal D})$. We deduce
	\begin{align*}
		\lambda_p(\bm{\mathcal D}) &=\dfrac{d_1}{2}
		\displaystyle\int_\Omega\displaystyle\int_\Omega J_1(x-y)\left(\varphi^p_1(x)-\varphi^p_1(y)\right)^2dxdy +\cdots+ \dfrac{d_N}{2}
		\displaystyle\int_\Omega\displaystyle\int_\Omega J_N(x-y)\left(\varphi^p_N(x)-\varphi^p_N(y)\right)^2dxdy\\
		&\quad +\displaystyle\int_\Omega\left[d_1\left(1 -k_1(x)\right)(\varphi^p_1)^2(x)+\cdots+d_N\left(1-k_N(x)\right)(\varphi^p_N)^2(x)\right]dx-\displaystyle\int_\Omega \varphi^p(x)^T A(x)\varphi^p(x) dx,
	\end{align*}
	where $k_i(x)=\displaystyle\int_\Omega J_i(x-y) dy$, $1\leq i\leq N$. Since $k_i(x)\leq 1$ for $x\in \Omega$ and $1\leq i\leq N$, it follows that
	$\lambda_p\left(\bm{\mathcal D}\right)$ is a monotone increasing function in $\bm{\mathcal D}$, which completes the proof \eqref{1}.

\noindent (2) First of all, in order to prove \eqref{2}, we consider the following operator
	$$
	\displaystyle L^0_\Omega[\psi]:=\int_\Omega J(x-y)\psi(y) dy-\psi(x), \psi\in \mathbf C(\Omega).
	$$
Then, by Theorem \ref{existence-eigen} and  \cite[Proposition $3.4$]{SX} the principal eigenvalue of $L^0_\Omega$, say $\lambda^0<0$  exists and is negative. Let $\psi^0\in \mathbf C(\overline{\Omega})\cap \mathbf E^{++}$ be an associated eigenfunction. 	
	
Now let us denote by $\displaystyle \lambda_{\bm{\mathcal D}}:=-d \lambda^0-N \sup_{x\in \overline{\Omega}}\left\{ \dfrac{\max_{1\leq j\leq N}\psi_j^0(x)}{\min_{1\leq j\leq N}\psi_j^0(x)}\max_{1\leq i,j\leq N}|a_{ij}(x)|\right\}$, where $d=\min\{d_1,\ldots, d_N\}$. Then, one sees that
$$
(\bm{\mathcal K} + \lambda_{\bm{\mathcal D}}\bm{\mathcal I})[\psi^0]= (\bm{\mathcal D} L^0_\Omega[\psi^0]+\bm{\mathcal A}+ \lambda_{\bm{\mathcal D}}\bm{\mathcal I})[\psi^0]  = \Big[\bm{\mathcal D} \lambda^0+\bm{\mathcal A} -d \lambda^0-N\dfrac{\max_{1\leq j\leq N}\psi_j(x)}{\min_{1\leq j\leq N}\psi_j(x)}\max_{x\in \overline{\Omega}}\|A(x)\|\Big][\psi^0]\leq 0,
$$	
where we use the following estimate
\begin{equation*}
\begin{split}
\left|\sum_{j=1}^N a_{ij}(x)\psi_j^0(x)\right|&\leq N \sup_{x\in \overline{\Omega}}\left\{ \max_{1\leq j\leq N}\psi_j^0(x)\max_{1\leq i,j\leq N}|a_{ij}(x)|\right\}\\
&\leq N \sup_{x\in \overline{\Omega}}\left\{ \dfrac{\max_{1\leq j\leq N}\psi_j^0(x)}{\min_{1\leq j\leq N}\psi_j^0(x)}\max_{1\leq i,j\leq N}|a_{ij}(x)|\right\}\psi_i^0(x)
\end{split}
\end{equation*}
for all $1\leq i\leq N$ and $x\in \overline{\Omega}$. Therefore, $(\lambda_{\bm{\mathcal D}}, \psi^0)$	 is a test pair for $\lambda_p(\bm{\mathcal D})$ and hence $\lambda_p(\bm{\mathcal D})\geq \lambda_{\bm{\mathcal D}}$. Letting $d\to +\infty$, we finally arrive at \eqref{2}. 
		
\noindent (3) From the properties of the kernel function $J_i,\ 1\leq i\leq N$ we obtain
	\begin{align*}
		\displaystyle \int_\Omega \displaystyle \int_\Omega J_i(x-y)\varphi_i(y)\varphi_i(x)dxdy\leq \displaystyle \int_\Omega \displaystyle \int_\Omega J_i(x-y)\left(\dfrac{\varphi_i^2(y)+\varphi_i^2(x)}{2}\right)dxdy\leq  \displaystyle \int_\Omega \varphi_i^2(x)dx, 1\leq i\leq N.
	\end{align*}
	For all ${\varphi}=
	\left(\varphi_1,\ldots, \varphi_n\right)\in \mathbf{E}$ and $\left\|{\varphi}\right\|_{\mathbf{E}}=1$, then 
	\begin{align*}
		\left\langle\bm{\mathcal K}{\varphi},{\varphi}\right\rangle=&d_1\displaystyle \int_\Omega \displaystyle \int_\Omega J_1(x-y)\varphi_1(x)\varphi_1(y)dxdy +\cdots+d_n\displaystyle \int_\Omega \displaystyle \int_\Omega J_n(x-y)\varphi_n(x)\varphi_n(y)dxdy\\
		&+\displaystyle \int_\Omega \varphi(x)^T(A(x)-D)\varphi(x)dx\\
		&\leq d_1\displaystyle \int_\Omega \varphi_1^2(x)dx+\cdots+d_n\displaystyle \int_\Omega \varphi_N^2(x)dx+\displaystyle \int_\Omega \varphi(x)^T(A(x)-D)\varphi(x)dx\\
		&=\displaystyle \int_\Omega \varphi(x)^T A(x)\varphi(x)dx\\
		&\leq\displaystyle \int_\Omega \bar\lambda(A(x))\|\varphi(x)\|^2dx\\
		&\leq  \sup_{x\in \Omega}\bar\lambda(A(x))\displaystyle \int_\Omega  \|\varphi(x)\|^2dx\\
		& =  \sup_{x\in \Omega}\bar\lambda(A(x)).
	\end{align*}
	Therefore, we conclude that
	\begin{align*}
		\lambda_p(\bm{\mathcal D}) = -\sup\limits_{\left\|{\varphi}\right\|_{\mathbf E} =1}\left\langle\bm{\mathcal K}{\varphi}, {\varphi}\right\rangle\geq - \sup_{x\in \Omega}\bar\lambda(A(x)).
	\end{align*}
	
	On the other hand, by Theorem \ref{existence-eigen}, one has
	\begin{align}\label{eq:3.22}
		\lambda_p(\bm{\mathcal D})\leq -\sup_{x\in \mathbb R^n}\bar\lambda(A(x)-D).
	\end{align}
	Hence, 	$\lambda_p(\bm{\mathcal D})\to  - \sup_{x\in \Omega}\bar\lambda(A(x))$ as $\bm{\mathcal D}\to 0$.

\end{proof}

\section{Asymptotic behaviour of the principal eigenvalue under scaling}\label{bcv-section-scal}
Let us recall that $\bm{\mathcal N}\colon \mathbf{E}\rightarrow \mathbf{C}$ given by
	\begin{align*}
	\left(\bm{\mathcal N}\varphi\right)(x)=\mathrm{diag}\left({\mathcal N}_1[\varphi_1](x),\ldots,\; {\mathcal N}_N[\varphi_N](x)\right),
	\end{align*}
	where ${\mathcal N}_i[\varphi_i](x): = 
	\displaystyle\int_\Omega J_i(x-y)\varphi_i(y)dy, \;\text{for}\; i=1,\ldots, N$. In what follows, $\bm{\mathcal N}_\delta$ denotes the $\bm{\mathcal N}$ associated to $J_{\sigma,i}$, $1\leq i\leq N$ and $\Omega_\delta$, where $J_{\sigma, i}(x-y):=\frac{1}{\sigma^N} J_i(\frac{x-y}{\sigma})$ and $\O_\sigma:=\frac{1}{\sigma}\O$.
	
	Now we start with the scaling invariance of $\bm{\mathcal N}+\bm{\mathcal A}$.
\subsection{Scaling invariance}

This invariance is a consequence of the following observation. By definition  of  $\lambda_p(\bm{\mathcal N}+ \bm{\mathcal A})$, we have for all $\lambda<\lambda_p(\bm{\mathcal N}+ \bm{\mathcal A})$, 
$$\bm{\mathcal N}[\varphi](x)+(A(x)+\lambda I)\varphi(x) \le 0 \quad \text{ in }\quad \O,$$
for some positive $\varphi \in \mathbf{E}(\Omega)$.
Let us denote by  $ X=\sigma x$, $\O_\sigma:=\frac{1}{\sigma}\O$, $\psi(X):=\varphi(\sigma X)$, $J_{\sigma, i}(x-y):=\frac{1}{\sigma^N} J_i(\frac{x-y}{\sigma})$, and $A_\sigma(x):=A\left(\frac{x}{\sigma}\right)$. Then,  by argument as in \cite[Subsection $4.1$]{BCV1}, $\psi \in \mathbf{E}(\Omega_\delta)$ satisfies
$$
\bm{\mathcal N}_{\sigma}[\psi](x)+(A_\sigma(x)+\lambda I)\psi(x)\le 0 \quad \text{in}\quad \O_\sigma.
 $$
Therefore, $\lambda \le \lambda_p(\bm{\mathcal N}_{\sigma}+A_\sigma)$ and as a consequence 
$$ \lambda_p(\bm{\mathcal N}+ \bm{\mathcal A})\le\lambda_p(\bm{\mathcal N}_{\sigma}+\bm{\mathcal A}_\sigma).$$
Interchanging the role of $\lambda_p(\bm{\mathcal N}+ \bm{\mathcal A})$ and $\lambda_p(\bm{\mathcal N}_{\sigma}+A_\sigma)$ in the above argument yields
$$ \lambda_p(\bm{\mathcal N}+ \bm{\mathcal A})\geq \lambda_p(\bm{\mathcal N}_{\sigma}+\bm{\mathcal A}_\sigma).$$
Hence, we get
$$
\lambda_p(\bm{\mathcal N}+ \bm{\mathcal A})=\lambda_p(\bm{\mathcal N}_{\sigma}+\bm{\mathcal A}_\sigma).
$$

Now let us recall that
\begin{equation} \label{eq7-5-1}
	\begin{split}
	&\displaystyle\lambda_v(\bm{\mathcal N}+ \bm{\mathcal A})=\inf_{\varphi\in \mathbf E(\Omega),\varphi\not \equiv 0} -\dfrac{<\bm{\mathcal N}+ \bm{\mathcal A}\varphi,\varphi>}{\| \varphi \|^2_{\mathbf E(\Omega)}   }\\
	&=\displaystyle \inf_{\varphi\in \mathbf E(\Omega),\varphi\not \equiv 0}
	\dfrac{\dfrac{1}{2} \sum\limits_{i=1}^N\int\limits_{\Omega}\int\limits_{\Omega}J_i(x-y)[\varphi_i(x)-\varphi_i(y)]^2 dxdy-\int\limits_{\Omega}\varphi(x)^T A(x) \varphi(x) -\sum\limits_{i=1}^N \int_{\Omega} \int\limits_{\Omega}J_i(x-y)  \varphi_i^2(x) dx dy }{\| \varphi \|^2_{\mathbf E (\Omega)}   },\\
&\lambda_p(\bm{\mathcal N}+ \bm{\mathcal A})=\sup\{\lambda \in \mathbb R\colon \exists \varphi \in \mathbf E (\overline{\Omega}), \varphi>0, \bm{\mathcal N}[\varphi]+A(x)\varphi(x)+\lambda \varphi\leq 0 \text{ in }\Omega\};\\
&\lambda_p'(\bm{\mathcal N}+ \bm{\mathcal A})=\inf\{\lambda \in \mathbb R\colon \exists \varphi \in \mathbf E(\Omega)\cap \mathbf E^\infty(\Omega), 0\not \equiv\varphi\geq 0, \bm{\mathcal N}[\varphi]+A(x)\varphi(x)+\lambda \varphi\geq 0\text{ in }\Omega\}.
\end{split}
\end{equation}
First of all, by Theorem \ref{existence-eigen} and Proposition \ref{proposition_3.4}, we obtain the following lemma.
\begin{lemma} \label{estimate-eigenvalue}We have
	\begin{align}\label{eq7-5-2}
		-\sup_{\Omega}\bar \lambda(A(x))-\max_{1\leq i\leq N} \sup_{\Omega} \int_\Omega J_i(x-y) dy\leq \lambda_v(\bm{\mathcal N}+ \bm{\mathcal A})<  -\sup_{\Omega}\bar \lambda(A(x))
	\end{align}
\end{lemma}

\begin{proof}
	Following the proof of Theorem \ref{existence-eigen}, we obtain the right-hand side equality in (\ref{eq7-5-2})
	$$
	\lambda_v(\bm{\mathcal N}+ \bm{\mathcal A})<  -\sup_{\Omega}\bar \lambda(A(x)).
	$$
	On the other hand, a computation shows that
	$$
	\dfrac{\int_{\Omega}\varphi(x)^T A(x) \varphi(x) dx}{\|\varphi \|^2_{E}}\leq  \sup_{\Omega} \bar\lambda(A(x))
	$$	
	and
	$$	
	\left|\sum_{i=1}^N\int_{\Omega} \int_{\Omega}J_i(x-y) \varphi_i^2(x) dx dy  \right|\leq \max_{1\leq i\leq N} \sup_{x\in \Omega} \int_{\Omega}  J_i(x-y)dy \int_{\Omega}\sum_{i=1}^N\varphi_i^2(x) dx dy\leq  \max_{1\leq i\leq N} \sup_{x\in \Omega} \int_{\Omega} J_i(x-y)dy\ \|\varphi \|^2_{\mathbf E}.
	$$
	Therefore, by (\ref{eq7-5-1}) we conclude that 
	$$
	-\sup_{\Omega}\bar \lambda(A(x))-\max_{1\leq i\leq N} \sup_{\Omega} \int_\Omega J_i(x-y) dy\leq \lambda_v(\bm{\mathcal N}+ \bm{\mathcal A}),
	$$	
	which completes the proof.
\end{proof}

Furthermore, by Lemma \ref{lem3.5-5}, one has the following lemma.
\begin{lemma}\label{equality-eigenvalue} One has
\begin{align}\label{eq7-5-3}
	\lambda_p'(\bm{\mathcal N}+ \bm{\mathcal A})=\lambda_p(\bm{\mathcal N}+ \bm{\mathcal A})=\lambda_v(\bm{\mathcal N}+ \bm{\mathcal A}).
\end{align}
\end{lemma}
\subsection{Asymptotic behaviors of $\lambda_p\left({\mathcal K}_{\sigma, m,\Omega}+{\mathcal A}\right)$ }

Let us  focus on the behaviour of the principal eigenvalue of  the spectral problem
$$
\bm{\mathcal K}_{\sigma, m,\Omega}[\varphi]+(\bm{\mathcal A}+\lambda \bm{\mathcal I})\varphi=0 \quad \text{ in }\quad \O,
$$
where $\bm{\mathcal K}_{\sigma, m,\Omega}[\varphi]=({\mathcal K}_{\sigma, m,\Omega}^1[\varphi_1],\ldots, {\mathcal K}_{\sigma, m,\Omega}^N[\varphi_N])$ with
$$
{\mathcal K}_{\sigma, m,\Omega}^i[\varphi_i]:=\frac{1}{\sigma^m} \left(\int_{\Omega}J_{\sigma,i}(x-y)\varphi_i(y)\,dy -\varphi_i(x)\right),\; i=1,\ldots, N.
$$ 
where 
$J_{\sigma,i}(z):=\frac{1}{\sigma^N}J_i\left(\frac{z}{\sigma}\right)$.  

\begin{lemma} If $\Omega_1\subset \Omega_2$, then $\lambda_p(\bm{\mathcal K}_{\sigma, m,\Omega_1})\geq \lambda_p(\bm{\mathcal K}_{\sigma, m,\Omega_2})$. In addition, one has
$$
\left|\lambda_p(\bm{\mathcal K}_{\sigma, m,\Omega_1})- \lambda_p(\bm{\mathcal K}_{\sigma, m,\Omega_2})\right|\leq C_0 |\Omega_2\setminus \Omega_1|,
$$
where $|\Omega|$ denotes the Lebesgue measure of $\Omega$ and $C_0$ is a positive constant depending on $J, \sigma, m$, and $\Omega_2$ .
\end{lemma}
\begin{proof} Let $(\lambda, \psi)$ be a test pair for $\lambda_p(\bm{\mathcal K}_{\sigma, m,\Omega_2})$, i.e., $(\lambda, \psi)\in \mathbb R\times \mathbf E^{++}(\Omega_2)$ satisfies 
$$
\left(\bm{\mathcal K}_{\sigma, m,\Omega_2} +\lambda\bm{\mathcal I}\right)[\psi]\leq 0\text{ on } \overline{\Omega}_2.
$$
Now let us define $\tilde \psi(x)=\psi(x)$ for all $x\in \overline{\Omega}_1$. Then, one has $\tilde\psi\in \mathbf E^{++}(\Omega_1)$. Moreover, for any $x\in \overline{\Omega}_1$ and $1\leq i\leq N$ we have
\begin{equation*}
\begin{split}
\left(\bm{\mathcal K}_{\sigma, m,\Omega_1} +\lambda\bm{\mathcal I}\right)[\tilde \psi]_i(x)&\leq \frac{1}{\sigma^m} \left(\int_{\Omega_1}J_{\sigma,i}(x-y)\tilde\psi_{i}(y)\,dy -\psi_{1i}(x)\right)+ \lambda \tilde\psi_{i}(x) \\
&\leq \frac{1}{\sigma^m} \left(\int_{\Omega_2}J_{\sigma,i}(x-y)\psi_{i}(y)\,dy -\psi_{i}(x)\right)+ \lambda \psi_{i}(x)  \\
&\leq \left(\bm{\mathcal K}_{\sigma, m,\Omega_2} +\lambda\bm{\mathcal I}\right)[\psi]_i(x)\leq 0.
\end{split}
\end{equation*}

Therefore, it follows that $(\lambda, \tilde\psi)$ is a test pair for $\lambda_p(\bm{\mathcal K}_{\sigma, m,\Omega_1})$. Thus $\lambda\leq \lambda_p(\bm{\mathcal K}_{\sigma, m,\Omega_1})$ and hence taking the supremum over all such $\lambda$, we obtain that
\begin{equation}\label{eq-18-1}
\lambda_p(\bm{\mathcal K}_{\sigma, m,\Omega_2})\leq \lambda_p(\bm{\mathcal K}_{\sigma, m,\Omega_1}).
\end{equation}

Now we are going to prove the other statement. Indeed, let $(\lambda_p(\bm{\mathcal K}_{\sigma, m,\Omega_2}), \psi)$ be the eigen-pair of $\bm{\mathcal K}_{\sigma, m,\Omega_2}$ with the normalization $\sup\limits_{\overline{\Omega}_2}\psi=1$. Then, a computation shows that
\begin{equation*}
\begin{split}
\bm{\mathcal K}_{\sigma, m,\Omega_1}[\psi]_i+\lambda_p(\bm{\mathcal K}_{\sigma, m,\Omega_2})[\psi]_i&= -\frac{1}{\sigma^m}\int\limits_{\Omega_2\setminus \Omega_1} J_{\sigma,i}(\cdot -y)\psi_i(y) dy\\
&\geq -\frac{\|J_{\sigma,i}\|_\infty}{\sigma^m} |\Omega_2\setminus\Omega_1|\\
&\geq -\frac{\|J_{\sigma,i}\|_\infty}{\sigma^m\min\limits_{\overline{\Omega}_1}\psi_i} |\Omega_2\setminus\Omega_1|\psi_i \;\text{ on }\; \overline{\Omega}_1
\end{split}
\end{equation*}
for $1\leq i\leq N$. This implies that
$$
\bm{\mathcal K}_{\sigma, m,\Omega_1}[\psi]+\left[\lambda_p(\bm{\mathcal K}_{\sigma, m,\Omega_2})+C_0|\Omega_2\setminus \Omega_1|\right][\psi]\geq 0 \;\text{ on }\; \overline{\Omega}_1,
$$
where $C_0=\dfrac{1}{\sigma^m}\max\limits_{1\leq i\leq N}\dfrac{\|J_{\sigma,i}\|_\infty}{\min\limits_{\overline{\Omega}_1}\psi_i }$, and hence
$$
\lambda_p(\bm{\mathcal K}_{\sigma, m,\Omega_1})\leq \lambda_p(\bm{\mathcal K}_{\sigma, m,\Omega_2})+C_0|\Omega_2\setminus \Omega_1|.
$$
Therefore, this together with (\ref{eq-18-1}) implies our result.
\end{proof}

Assuming that $0\le m< 2$, we obtain here the limits of $\lambda_p(\bm{\mathcal K}_{\sigma,m}+ \bm{\mathcal A})$ when $\sigma \to 0$ and $\sigma\to \infty$. 
But before going to the study of these limits, we recall a known inequality (cf. \cite[Lemma $4.1$]{BCV1}).
\begin{lemma}[see Lemma $4.1$ in \cite{BCV1}]\label{bcv-lem-I-le-J}
	Let $J\in C(\mathbb R^n )$, $J\geq 0$, $J$ symmetric with unit mass, such that $|z|^2J(z) \in (L^1(\mathbb R^n ))^N$ . Then for all  $\varphi \in H_0^1(\Omega)$ we have
	$$-\int_{\Omega}\left(\int_{\Omega}J(x-y)\varphi(y)\,dy-\varphi(x)\right)\varphi(x)\,dx\le \frac{1}{2}\int_{\mathbb R^n }J(z)|z|^2\,dz \nlp{\nabla \varphi}{2}{\Omega}^2.$$ 
\end{lemma}

Let us also introduce  the following notation 

\begin{equation}
	\begin{split}
		&J_{\sigma,i}(z):=\frac{1}{\sigma^N}J_i\left(\frac{z}{\sigma}\right), \qquad \quad p_{\sigma,i}(x):=\int_{\Omega}J_{\sigma,i}(x-y)\,dy, \qquad \quad  D_2(J_i):=\int_{\mathbb R^n }J_{\sigma,i}(z)|z|^2\,dz, 1\leq i\leq N,\\
		&\mathcal{A}(\varphi):=\frac{\int_{\Omega}   \varphi(x)^T A(x)\varphi(x)\,dx} {\|\varphi\|^2_{\mathbf E}},\qquad \quad  \r_{\sigma,m}(\varphi):=\frac{1}{\sigma^m} \frac{\displaystyle\sum_{i=1}^N\int_{\Omega}(p_{\sigma,i}(x)-1)\varphi_i^2(x)\,dx} {\|\varphi\|^2_{\mathbf E}},\\
		&\mathcal{I}_{\sigma,m}(\varphi):=\frac{\frac{1}{\sigma^m}\left(-\displaystyle\sum_{i=1}^N\int_{\Omega}\left(\int_{\Omega}J_i(x-y)\varphi_i(y)\,dy-\varphi_i(x)\right)\varphi_i(x)\,dx\right)}{\|\varphi\|^2_{\mathbf E}} -\mathcal{A}(\varphi)\\
		&\j(\varphi):=\displaystyle\sum_{i=1}^N  \frac{D_2(J_i)}{2}\frac{\displaystyle\int_{\Omega}|\nabla \varphi_i|^2(x)\,dx} {\|\varphi\|^2_{\mathbf E}}.
	\end{split}
\end{equation}

With this notation, we see that  
$$\lambda_v(\bm{\mathcal K}_{\sigma,m,\Omega} + \bm{\mathcal A})=\inf_{\varphi \in \mathbf E} \mathcal{I}_{\sigma,m}(\varphi), $$

and by Lemma \ref{bcv-lem-I-le-J}, for any $\varphi \in H_0^1(\Omega)$ we get 

\begin{equation}\label{bcv-eq-I-le-J-R} 
	\mathcal{I}_{\sigma,m}(\varphi)\le \sigma^{2-m}\j(\varphi) -\mathcal A(\varphi).
\end{equation}

We are now in position to obtain the different limits of $\lambda_p(\bm{\mathcal K}_{\sigma,m,\Omega} + \bm{\mathcal A})$ as $\sigma \to 0$ and $\sigma\to \infty$. For simplicity,  we analyse three distinct situations: $m=0$ and $0<m<2$. 

Let us first deal with the easiest case, that is, when $0<m<2$.
\subsubsection{The case $0<m<2$:}\label{SS4.2.1}

\begin{proof}[\textbf{Proof of Theorem \ref{scaling limit}.}]The case $0<m<2$:
	
	First, let us look at the limit  of $\lambda_p$ when $\sigma \to 0$.  
	Up to adding a large positive constant to the function $a$, without any loss of generality,  we can assume that the function $a$ is positive somewhere in $\O$. 
	
	Since $\bm{\mathcal K}_{\sigma,m,\Omega}+A$ is a self-adjoined operator, by \eqref{eq7-5-3} and \eqref{bcv-eq-I-le-J-R}, for any $\varphi \in H^1_0(\Omega)$ we have
	$$\lambda_p(\bm{\mathcal K}_{\sigma,m,\Omega}+ \bm{\mathcal A})=\lambda_v(\bm{\mathcal K}_{\sigma,m,\Omega}+ \bm{\mathcal A})\le \mathcal{I}_{\sigma,m}(\varphi)\le \sigma^{2-m}\j(\varphi) -\mathcal A(\varphi).$$
	
	Define $\nu:=\sup_{\Omega} \bar \lambda(A(x))>0$, and let $(x_k)_{k\in \N^*}$  be a sequence of point such that $\bar \lambda(A(x_k))>\nu-\frac{1}{k}$. Since $A$ satisfy the condition ($P$), we can also assume that for all $k$, there exists $r_k>0$ such that
$\bar \lambda(A(x))>\nu-\frac{1}{k}$ for all $x\in B(x_k,r_k)$. Moreover, by Lemma \ref{keylemma-cutoff-function}, one can find a sequence of smooth functions $\{\varphi_k\}$ with $\supp \varphi_k \subset B(x_k, r_k)$ such that 
$$
\frac{ \int_{B(x_k,r_k)} \varphi_\rho^T(x)A(x)\varphi_\rho(x)\,dx}{\nlp{\varphi_\rho}{2}{\Omega}^2}\geq \nu-\frac{2}{k}, \; \; \forall k\in \mathbb N^*,
$$
 and  therefore, 	
$$
	\limsup_{\sigma\to 0} \lambda_p(\bm{\mathcal K}_{\sigma,m,\Omega}+ \bm{\mathcal A})\le -{\mathcal A}(\varphi_\rho)=-\frac{\int_{B(x_k,r_k)} \varphi_\rho^T(x)A(x)\varphi_\rho(x)\,dx}{\nlp{\varphi_\rho}{2}{\Omega}^2}\leq -\nu+\frac{2}{k} , \; \forall k\in \mathbb N^*.
$$  
Hence, by sending now $k \to \infty$ in the above inequality, we obtain	
$$
	\limsup_{\sigma\to 0} \lambda_p(\bm{\mathcal K}_{\sigma,m,\Omega}+ \bm{\mathcal A})\leq -\nu.
$$
	
	On the other hand, by using the test functions $(\varphi^i,\lambda)=({\bf e_i},-\nu)$ we can easily check that for any $\sigma >0$
	$$\lambda_p(\bm{\mathcal K}_{\sigma,m,\Omega}+ \bm{\mathcal A})\geq -\nu.$$
	(Here, we use the fact that $\| A(x) e_i\| \leq \bar\lambda(A(x)) \|e_i\|=\bar\lambda(A(x)) \leq \nu$.)
	Hence, 
	$$-\nu\le \liminf_{\sigma\to 0} \lambda_p(\bm{\mathcal K}_{\sigma,m,\Omega}+ \bm{\mathcal A})\leq \limsup_{\sigma\to 0} \lambda_p(\bm{\mathcal K}_{\sigma,m,\Omega}+ \bm{\mathcal A})\le -\nu.$$

	Now, let us  look at the limit of $\lambda_p(\bm{\mathcal K}_{\sigma,m,\Omega}+ \bm{\mathcal A})$ when $\sigma \to +\infty$. 
	This limit is a straightforward consequence of Lemma \ref{estimate-eigenvalue}.
	Indeed, as remarked above, for any $\sigma$ by using the test function $(\varphi^i,\lambda)=({\bf e_i},-\nu)$, we have
	$$-\nu \le \lambda_p(\bm{\mathcal K}_{\sigma,m,\Omega}+ \bm{\mathcal A})$$
	whereas from Lemma \ref{estimate-eigenvalue} we have 
	$$ \lambda_p(\bm{\mathcal K}_{\sigma,m,\Omega}+ \bm{\mathcal A})\leq -\sup_{\Omega}\left(\bar \lambda(A(x))\right)+\frac{1}{\sigma^m}. $$
	Therefore, since $m>0$ we have 
	
	$$-\nu \le \lim_{\sigma\to +\infty} \lambda_p(\bm{\mathcal K}_{\sigma,m,\Omega}+ \bm{\mathcal A})\le -\nu. $$
\end{proof}

\begin{remark}\label{bcv-pev-rem-asymp}.
	From the proof, we  obtain also some of the limits in the  cases  $m=0$ and $m=2$.  Indeed, the analysis of the limit of $\lambda_p(\bm{\mathcal K}_{\sigma,m,\Omega}+ \bm{\mathcal A})$ when $\sigma \to 0$ holds true as soon as $m<2$. Thus,  $$\lambda_p(\bm{\mathcal K}_{\sigma,0,\Omega}+ \bm{\mathcal A})\to -\sup_{\Omega}  \bar\lambda(A(x)) \quad\text{ as } \quad\sigma \to 0.$$
	On the other hand, the  analysis of the  limit of $\lambda_p(\bm{\mathcal K}_{\sigma,m,\Omega}+ \bm{\mathcal A})$ when $\sigma \to +\infty$ holds true as soon as $m>0$. Therefore, $$\lambda_p(\bm{\mathcal K}_{\sigma,2,\Omega}+ \bm{\mathcal A})\to -\sup_{\Omega} \bar\lambda(A(x)) \quad\text{ as } \quad\sigma \to +\infty.$$ 
	
\end{remark}

\subsubsection{The case $m=0$}
In this situation,   one cannot use above argument to obtain the  limits as $m=0$, new idea must be figured out to obtain the right limits. Indeed, we can prove:

\begin{proof}[\textbf{Proof of Theorem \ref{scaling limit}. }] The case $m=0$.\\

	As already noticed in Remark \ref{bcv-pev-rem-asymp}, the limit of $\lambda_p(\bm{\mathcal K}_{\sigma,0,\Omega}+ \bm{\mathcal A})$ when $\sigma \to 0$ can be obtained by following the arguments developed in the case $0<m<2$. Therefore, it remains only to establish the limit of $\lambda_p(\bm{\mathcal K}_{\sigma,0,\Omega}+ \bm{\mathcal A})$ when $\sigma \to \infty$.
	
As above, up to adding a large positive constant matrix to $A(x)$, without any loss of generality, we can assume that $A(x)$ is positive somewhere in $\O$ and we denote  $\nu:=\sup_{\Omega}\bar\lambda(A(x))>0$. 
	By using constant test functions and Lemma \ref{estimate-eigenvalue}, we observe that

	$$-\nu\le \lambda_p(\bm{\mathcal K}_{\sigma,0,\Omega}+ \bm{\mathcal A})\le 1 -\nu, \quad \text{ for all }\quad \sigma >0.$$
	So, we have $$\limsup_{\sigma\to \infty}\lambda_p(\bm{\mathcal K}_{\sigma,0,\Omega}+ \bm{\mathcal A})\le 1 -\nu.$$
	On the other hand, for any $\varphi \in \mathbf C_c(\Omega)$  we have for all $\sigma>0$,
	
	\begin{align*}
		&\mathcal{I}_{\sigma,0}(\varphi) \displaystyle \sum_{i=1}^N \int_{\Omega}\varphi_i^2(x)dx=-\displaystyle \sum_{i=1}^N\int_{\Omega}\left(\int_{\Omega} J_{\sigma,i}(x-y)\varphi_i(y)\,dy-\varphi_i(x)\right)\varphi_i(x)dx -\int_{\Omega} \varphi(x)^TA(x)\varphi(x)\,dx\\
		&=-\displaystyle \sum_{i=1}^N \iint_{\Omega\times \Omega} J_{\sigma,i}(x-y)\varphi_i(x)\varphi_i(y)dxdy +\displaystyle \sum_{i=1}^N \int_{\Omega}\varphi_i^2(x)\,dx -\int_{\Omega} \varphi(x)^TA(x)\varphi(x)\,dx\\
		&\geq-\displaystyle \sum_{i=1}^N  \nlp{\varphi_i}{2}{\Omega}\left(\int_{\Omega}\left( \int_{\Omega} J_{\sigma,i}(x-y)\varphi_i(x)\,dx\right)^2\,dy\right)^{1/2} +\displaystyle \sum_{i=1}^N \int_{\Omega}\varphi_i^2(x)\,dx -\sup_{\Omega}\bar\lambda(A(x))\displaystyle \sum_{i=1}^N\int_{\Omega}\varphi^2(x)\,dx\\
		&\geq-\displaystyle \sum_{i=1}^N\sqrt{\|J_{\sigma,i}\|_{\infty}}\nlp{\varphi_i}{2}{\Omega}^2+\displaystyle \sum_{i=1}^N\int_{\Omega}\varphi_i^2(x)\,dx -\nu\displaystyle \sum_{i=1}^N\int_{\Omega}\varphi_i^2(x)\,dx\\
		&\geq\left(-\frac{\max_{i=1,\ldots, N}\sqrt{\|J_i\|_{\infty}}}{\sigma^{N/2}}+1-\nu \right) \displaystyle \sum_{i=1}^N\int_{\Omega}\varphi_i^2(x)\,dx.
	\end{align*}
Thus, for all $\sigma>0$ we have 
	$$ \mathcal{I}_{\sigma,0}(\varphi)\geq \left(-\frac{\max_{i=1,\ldots, N}\sqrt{\|J_i\|_{\infty}}}{\sigma^{N/2}}+1-\nu \right).$$
	By density of $\mathbf C_c(\Omega)$ in  $L^2(\Omega)$, the above inequality holds  for any $\varphi \in \mathbf E$.
	
	Therefore, by (\ref{eq7-5-3}) for all $\sigma$
	$$
	\lambda_p(\bm{\mathcal K}_{\sigma,0,\Omega} + \bm{\mathcal A})=\lambda_v(\bm{\mathcal K}_{\sigma,0,\Omega} + \bm{\mathcal A})\geq -\frac{\sqrt{\|J\|_{\infty}}}{\sigma^{N/2}}+ 1 -\nu,
	$$
	and 
	$$ \liminf_{\sigma\to +\infty} \lambda_p(\bm{\mathcal K}_{\sigma,0,\Omega} + \bm{\mathcal A})\geq 1 -\nu.$$
	Hence,
	$$1 -\nu\le \liminf_{\sigma\to +\infty} \lambda_p(\bm{\mathcal K}_{\sigma,0,\Omega} + \bm{\mathcal A})\le \limsup_{\sigma\to +\infty} \lambda_p(\bm{\mathcal K}_{\sigma,0,\Omega} + \bm{\mathcal A}) \leq 1 -\nu.$$
\end{proof}

 \textbf{Conflicts of interest:} The authors have no conflicts of interest to declare that are relevant to the content of this article.

\end{document}